\documentclass[a4paper,12pt,reqno]{amsart}
\usepackage{graphicx}

\usepackage{color}

\usepackage{amsmath,amstext,amssymb,amsopn,amsthm}
\usepackage{url}

\usepackage[margin=30mm]{geometry}
\usepackage{eucal,mathrsfs,dsfont}
\usepackage{soul}

\newtheorem{theorem}{Theorem}[section]
\newtheorem{corollary}[theorem]{Corollary}
\newtheorem{lemma}[theorem]{Lemma}
\newtheorem{proposition}[theorem]{Proposition}

\theoremstyle{definition}

\newtheorem{definition}[theorem]{Definition}

\theoremstyle{remark}
\newtheorem{remark}[theorem]{Remark}
\newtheorem{example}[theorem]{Example}

\definecolor{mr}{rgb}{0.1,0.2,0.7}

\newcommand{\calB}{\mathcal{B}}

\newcommand{\WUSC}[3]{\textrm{\rm WUSC}(#1,#2,#3)}
\newcommand{\WLSC}[3]{\textrm{\rm WLSC}(#1,#2,#3)}
\newcommand{\lC}{{\underline{C}}}
\newcommand{\uC}{{\overline{C}}}
\newcommand{\la}{{\underline{\alpha}}}
\newcommand{\ua}{{\overline{\alpha}}}

\newcommand{\R}{\mathds{R}}

\newcommand{\N}{{\mathds{N}}}

\newcommand{\RR}{\mathrm{I\kern-0.20emR}}
\newcommand{\D}{\mathrm{d}\kern0.2pt}

\newcommand{\p}{\mathbb{P}}

\DeclareMathOperator{\dist}{dist}

\title[Gradient estimates of Dirichlet heat kernels for L{\'e}vy processes]{Gradient estimates of Dirichlet heat kernels for unimodal L{\'e}vy processes}
\author[T. Kulczycki]{Tadeusz Kulczycki}
\author[M. Ryznar]{Micha{\l} Ryznar}
\thanks{T. Kulczycki was supported in part by the National Science Centre, Poland, grant no. 2015/17/B/ST1/01233, M. Ryznar was supported in part by the National Science Centre, Poland, grant no. 2015/17/B/ST1/01043}
\address{Faculty of Pure and Applied Mathematics, Wroc{\l}aw University of Science and Technology, Wyb. Wyspia{\'n}skiego 27, 50-370 Wroc{\l}aw, Poland.}
\email{Tadeusz.Kulczycki@pwr.edu.pl}
\email{Michal.Ryznar@pwr.edu.pl}

\begin{document}
\begin{abstract} Under some mild assumptions on the L{\'e}vy measure and the symbol we obtain gradient estimates of Dirichlet heat kernels for pure-jump isotropic unimodal L{\'e}vy processes in $\R^d$.
\end{abstract}

\maketitle

\section{Introduction}

The Dirichlet heat kernels for L{\'e}vy processes have been intensively studied in recent years. Qualitatively sharp estimates for classical Dirichlet heat kernels for the Brownian motion were established in 2002 by Zhang \cite{Z2002} for $C^{1,1}$ domains and in 2003 by 
Varopoulos \cite{V2003} for Lipschitz domains. Upper bound for the Dirichlet heat kernels for the isotropic stable processes were given in 2006 by Siudeja \cite{S2006} for convex sets. He used some ideas from \cite{KS2006}. In 2008 Chen, Kim and Song \cite{CKS2008} obtained sharp, two sided estimates for the Dirichlet heat kernels for the isotropic stable process for $C^{1,1}$ open sets. In 2009 Bogdan, Grzywny and Ryznar \cite{BGR2009} showed similar results for $\kappa$-fat open sets. Gradual extensions were then obtained for some subordinate Brownian motions \cite{CKS2012,CKS2014,CKS2012a}, for L{\'e}vy processes with comparable L{\'e}vy measure \cite{KSV2014} and unimodal L{\'e}vy processes satisfying some scaling conditions \cite{BGR2014a}.

The aim of this paper is to obtain gradient estimates of Dirichlet heat kernels for unimodal L{\'e}vy processes whose  symbols
satisfy some scaling conditions and L{\'e}vy measures satisfy some regularity conditions. The main result is the following theorem.
\begin{theorem}
\label{main1}
Let $X = (X_t, t \ge 0)$ be a pure-jump isotropic L{\'e}vy process in $\R^d$ with the characteristic exponent $\psi$, which satisfies WLSC($\underline{\alpha},\theta_0,\underline{C}$) and WUSC($\overline{\alpha},\theta_0,\overline{C}$) for some $\underline{\alpha} > 0$, $\overline{\alpha} \in (0,2)$, $\theta_0 \ge 0$, and $\underline{C}, \overline{C}  > 0$. We assume that the L{\'e}vy measure of $X$ is infinite and has the strictly positive density $\nu(x) = \nu(|x|)$, $\nu(r)$ is nonincreasing, absolutely continuous such that $-\nu'(r)/r$ is nonincreasing,  it satisfies $\nu(r) \le a \nu(r+1), r \ge 1$ for some constant $a$.

Let $D \subset \R^d$ be an open, nonempty set and $p_D(t,x,y)$ be the Dirichlet heat kernel for $X$ on $D$. Then $\nabla_x \, p_D(t,x,y)$ exists for any $x, y \in D$, $t > 0$ and we have
\begin{equation}
\label{maineq}
|\nabla_x \, p_D(t,x,y)| \le c \left[\frac {1}{\delta_D(x)\wedge1} \vee \psi^{-}(1/t)\right] p_{D}(t, x, y), \quad \quad x, y \in D, \,\,  t \in (0,1],
\end{equation} 
where $c = c(d,\psi)$, $\delta_D(x) = {\dist(x, D^c)}$ and $\psi^-$ denotes the generalized inverse of $\psi^*(r)=sup_{\rho \in [0,r]}\psi(\rho)$.
\end{theorem}
The notation used in the formulation of the above theorem is explained in Preliminaries.

By the semigroup property one easily gets
\begin{corollary}
\label{largetime}
If $X$ and $D$ satisfy assumptions of Theorem \ref{main1} then we have
\begin{equation*}
|\nabla_x \, p_D(t,x,y)| \le c \left[\frac {1}{\delta_D(x)\wedge1} \vee \psi^{-}(1)\right] p_{D}(t, x, y), \quad \quad x, y \in D, \,\,  t \in [1,\infty),
\end{equation*} 
where $c = c(d,\psi)$.
\end{corollary}

Analogous gradient estimates of the classical Dirichlet heat kernel were obtained in 2006 by Zhang \cite[Theorem 2.1]{Z2006}. Note that Zhang's estimates have different shape than estimates obtained in Theorem \ref{main1}. Namely, in Zhang's estimate there is an additional term $ |x-y|/\sqrt{t}$. This is caused by a different behaviour of Dirichlet heat kernels for pure-jump L{\'e}vy processes and the classical one.  

The Dirichlet semigroup of the process $X$ on the domain $D \subset \R^d$ is given by 
$$
P_t^D f(x) = \int_D p_D(t,x,y) f(y) \, dy, \quad \quad x \in D, \, t > 0.
$$

Directly from Theorem \ref{main1} we get
\begin{corollary}
\label{parabolic}
Let $X$, $D$ satisfy assumptions of Theorem \ref{main1} and $f$ be nonnegative, measurable and bounded on $\R^d$. Then $\nabla_x \, P_t^D f(x)$ exists for any $t > 0$, $x \in D$ and we have
\begin{equation*}
|\nabla_x \, P_t^D f(x)| \le c \left[\frac {1}{\delta_D(x)\wedge1} \vee \psi^{-}(1/t)\right] P_t^D f(x), \quad \quad x, y \in D, \,\,  t \in (0,1],
\end{equation*} 
where $c = c(d,\psi)$.
\end{corollary}
Note that we have a simple estimate $P_t^D f(x) \le P^x(\tau_D > t) \|f\|_\infty$ and for bounded, sufficiently smooth ($C^{1,1}$) domains $D$ sharp estimates of $P^x(\tau_D > t)$ are known \cite[Theorem 4.5]{BGR2014a}.

Gradient estimates of $P_t^D f$ for Dirichlet semigroups corresponding to diffusion processes $X$ whose generators are second-order elliptic operators have been intensively studied see e.g. \cite{W2004, FMP2004}, \cite[Chapter 3]{L1995}. The main motivation of such estimates comes from well-known connections with stochastic differential equations see e.g. \cite{H1980, V1980}.

Estimates of $|\nabla_x \, P_t f(x)|$, where $\{P_t\}_{t > 0}$ is the semigroup of a free L{\'e}vy process were obtained in 2012 by Schilling, Sztonyk, Wang \cite[Theorem 1.3]{SSW2012}. $P_t f$ is given by $P_t f(x) = \int p_t(x-y) f(y) \, dy$, where $p_t(x)$ is the transition density of the process $X$.  Interesting estimates of derivatives of $p_t$ were obtained in \cite{SSW2012} and \cite{KS2015}. The gradient estimates $|\nabla_x \, P_t f(x)|$ from \cite{SSW2012} were applied to study stochastic differential equations driven by L{\'e}vy processes in \cite{M2015}. 

In 2011 estimates of $\|D P_tf\|$ were obtained by Priola and Zabczyk in \cite[Theorem 4.14, (5.19)]{PZ2011} for the semigroup $\{P_t\}_{t > 0}$ of the solution to a nonlinear stochastic partial differential equation
$$
d \, X_t = A X_t \, dt + F(X_t) \, dt + d Z_t
$$
in a real seperable Hilbert space driven by an infinite dimensional, cylindrical stable process $Z$ ($D$ is  here a Fr{\'e}chet derivative). The results of that paper apply to stochastic heat equation with Dirichlet boundary conditions, see \cite[Example 5.5]{PZ2011}. It is also worth  pointing  out that Da Prato, Goldys and Zabczyk in \cite[Corollary on page 437]{DGZ1997} obtained estimates of $|D P_t^{\mathcal{O}} f(x)|$ where $\{P_t^{\mathcal{O}}\}_{t > 0}$ is the Dirichlet semigroup of the process $X$ on ${\mathcal{O}}$,  $X$ is the solution of a linear stochastic equation
$$
d \, X = A X \, dt + d\, W
$$
in a separable Hilbert space $H$ driven by a Wiener process $W$ on $H$, and ${\mathcal{O}}$ is an open subset of $H$.

The paper is organized as follows. In Section 2 we introduce the notation and collect known facts  needed in the sequel. In Section 3 we provide  some auxiliary estimates of the heat kernel and the L{\'e}vy measure. The main section of the paper is Section 4 in which we prove Theorem \ref{main1}. The next section contains examples of processes which satisfy assumptions of Theorem \ref{main1}. In Appendix we provide a proof of Theorem \ref{hk_kula2} which is an extension of Theorem 4.5 in \cite{BGR2014a}.

\section{Preliminaries}

In the whole paper we  use a convention that for a radial function $f:\R^d \to \R$ we  write $f(x) = f(r)$, if $x \in \R^d$ and $|x| = r$. All constants appearing in this paper are positive and finite. We write $\kappa = \kappa(a,\ldots,z)$ to emphasize that $\kappa$ depends only on $a,\ldots,z$. We adopt the convention that constants denoted by $c$ (or $c_1$, $c_2$) may change their value from one use to another. We write $f(x) \approx g(x)$ for $x \in A$ and say $f$ and $g$ are comparable for $x \in A$ if $f,g \ge 0$ on $A$ and there is a number $c \ge 1$, called comparability constant, such that $c^{-1} f(x) \le g(x) \le c f(x)$ for $x \in A$. For $x \in \R^d$ and $r > 0$ we let $B(x,r) = \{y \in \R^d: \, |y - x| < r\}$. We denote $a \wedge b = \min(a,b)$ and $a \vee b = \max(a,b)$ for $a, b \in \R$. When $D \subset \R^d$ is an open set we denote by $\mathcal{B}(D)$ a family of Borel subsets of $D$.

A Borel measure on $\R^d$ is called isotropic unimodal if on $\R^d \setminus \{0\}$ it is absolutely continuous with respect to the Lebesgue measure and has a finite, radial, radially nonincreasing density function. Such measures may have an atom at the origin. A L{\'e}vy process $X=(X_t, t \ge 0)$ is called isotropic unimodal if its transition probability $p_t(dx)$ is isotropic unimodal for all $t > 0$. Unimodal isotropic pure-jump L\'evy processes are characterized \cite{W1983} by unimodal isotropic L{\'e}vy measures $\nu(dx) = \nu(x) \, dx = \nu(|x|) \, dx$. 

The characteristic exponent of $X$ is given by
$$
\psi(\xi) = \int_{\R^d} \left(1 - \cos\langle\xi,x\rangle\right) \, \nu(dx), \quad \xi \in \R^d.
$$

In the whole paper we assume that $X$ is a pure-jump isotropic unimodal L{\'e}vy process in $\R^d$ with the characteristic exponent $\psi$ and that the L{\'e}vy measure of $X$ is infinite. We also assume that the following Hartman-Wintner condition holds
\begin{equation}
\label{logpsi}
\lim_{|x| \to \infty} \frac{\psi(x)}{\log |x|} = \infty.
\end{equation}
This guarantees that for any $t > 0$,   $p_t(dx)$ has a radial, radially nonincreasing density function $p_t(x)$, which
 is bounded and smooth on $\R^d$.


The derivative $\nu'(r)$ is understood as a function (defined a.e. on $(0, \infty)$) such that  $\nu(r) = -\int_r^\infty\nu'(\rho) d\rho,\,  r>0 $.
In fact, under the assumption that $-\nu'(r)/r$ is nonincreasing on the set where it is
defined, we can always take a version which is well defined for each point $r > 0$ and
$-\nu'(r)/r$ is nonincreasing on $(0, \infty)$. Throughout the whole paper we use that meaning
of $\nu'(r)$.
Note also that if $\nu(r)$ is convex then $-\nu'(r)/r$ is nonincreasing (in the above sense).
 
Now we recall the definition of scaling conditions (cf. \cite{BGR2014}). Let $\varphi$ be a non-negative, non-zero function on $[0,\infty)$. We say that $\varphi$ satisfies {\it{a weak lower scaling condition}} WLSC($\underline{\alpha},\theta_0,\underline{C}$) (and write $\varphi \in \text{WLSC}(\underline{\alpha},\theta_0,\underline{C}$) or $\varphi \in \text{WLSC}$) if there are numbers $\underline{\alpha} > 0$, $\theta_0 \ge 0$ and $\underline{C} > 0$ such that
$$
\varphi(\lambda \theta) \ge \underline{C} \lambda^{\underline{\alpha}} \varphi(\theta), \quad \text{for} \quad \lambda \ge 1, \, \theta \ge \theta_0.
$$
We say that $\varphi$ satisfies {\it{a weak upper scaling condition}} WUSC($\overline{\alpha},\theta_0,\overline{C}$) (and write $\varphi \in \text{WUSC}(\underline{\alpha},\theta_0,\underline{C}$) or $\varphi \in \text{WUSC}$) if there are numbers $\overline{\alpha} \in (0,2)$, $\theta_0 \ge 0$ and $\overline{C} > 0$ such that
$$
\varphi(\lambda \theta) \le \overline{C} \lambda^{\overline{\alpha}} \varphi(\theta), \quad \text{for} \quad \lambda \ge 1, \, \theta \ge \theta_0.
$$
Note that the condition $\psi \in \text{WLSC}(\underline{\alpha},\theta_0,\underline{C})$ implies (\ref{logpsi}).

Recall that the maximal characteristic function is defined by $\psi^*(r)=sup_{\rho \in [0,r]} \psi(\rho)$, $r \ge 0$. We define its generalized inverse $\psi^-: [0,\infty) \to [0,\infty]$ by
$$
\psi^-(x) = \inf\{y \ge 0: \, \psi^*(y) \ge x\}, \quad 0 \le x < \infty,
$$
with the convention that $\inf \emptyset = \infty$.
It is well known \cite[Proposition 2]{BGR2014} (see also \cite[Proposition 1]{G2013})
\begin{equation}
\label{psi*}
\psi(r) \le \psi^*(r) \le \pi^2 \psi(r), \quad r \ge 0.
\end{equation}

In the paper we will use the renewal function $V$ of the properly normalized ascending ladder-height process of $X_t^{(1)}$, where $X_t^{(1)}$ is the first coordinate of $X_t$. The ladder-height process is a subordinator with the Laplace exponent
$$
\kappa(\xi) = \exp\left\{\frac{1}{\pi} \int_0^{\infty} \frac{\log \psi(\xi \zeta)}{1 + \zeta^2} \, d\zeta\right\}, \quad \xi \ge 0,
$$
and $V(x)$ is its potential measure of the half-line $(-\infty,x)$. The Laplace transform of $V$ is given by
$$
\int_0^{\infty} V(x) e^{-\xi x} \, dx = \frac{1}{\xi \kappa(\xi)}, \quad \xi > 0.
$$ 
For a detailed discussion of the properties of $V$ we refer the reader to \cite{S1980}. We have $V(x) = 0$ for $x \le 0$ and $V(\infty):= \lim_{r \to \infty} V(r) = \infty$. $V$ is subadditive, that is 
$$
V(x + y) \le V(x) + V(y), \quad x,y \in \R.
$$
It is known that $V$ is absolutely continuous on $(0,\infty)$ and strictly increasing on $(0,\infty)$. 
We will use $V$ and its inverse function $V^{-1}$ in the estimates of heat kernels. 
By \cite[Proof of Proposition 2.4]{BGR2015} we have
\begin{equation}
\label{Vpsi}
c_1 \psi\left(\frac{1}{r}\right) \le \frac{1}{V^2(r)} \le c_2 \psi\left(\frac{1}{r}\right), \quad r > 0,
\end{equation}
where $c_1$, $c_2$ are absolute constants. It is clear that 
\begin{equation}
\label{Vr}
r < V^{-1}(\sqrt{t}) \iff V^2(r) < t, \quad r, t > 0,
\end{equation}
and 
\begin{equation}
\label{rVd}
r < V^{-1}(\sqrt{t}) \iff \frac{t}{V^2(r) r^d} > [V^{-1}(\sqrt{t})]^{-d}, \quad r, t > 0.
\end{equation}

If $\psi \in \text{WLSC}$ and $\psi \in \text{WUSC}$ then by (\ref{psi*}), (\ref{Vpsi}) we get for any $t \in (0,1]$
\begin{equation}
\label{Vpsi-1}
\frac{1}{V^{-1}(\sqrt{t})} \approx \psi^{-}\left(\frac{1}{t}\right),
\end{equation}
where the comparability constant depends only on $\psi$.

By (\ref{Vpsi}), \cite[(1.9)]{BGR2014a}, \cite[Remark 1.4]{BGR2014a} and arguments similar to the justification of \cite[(1.8)]{BGR2014a} we obtain the following scaling properties 
 $V^{-1}$. 
\begin{lemma}
\label{Vscaling}
Let $\psi$ satisfy WLSC($\underline{\alpha},\theta_0,\underline{C}$) and WUSC($\overline{\alpha},\theta_0,\overline{C}$) for some $\underline{\alpha} > 0$, $\overline{\alpha} \in (0,2)$, $\theta_0 \ge 0$, and $\underline{C}, \overline{C}  > 0$. Then there exists $c_1 = c_1(d,\psi)$ such that
$$
V^{-1}(\eta \omega) \ge c_1 \eta^{2/\underline{\alpha}} V^{-1}(\omega), \quad \text{for} \quad \eta \in (0,1], \, \omega \in (0,1],
$$
\end{lemma}

Now we introduce the condition (\textbf{H}), the reader is referred to \cite{BGR2014a} for a detailed exposition.
\begin{definition}
\label{conditionH}
We say that condition (\textbf{H}) holds if for every $r > 0$ there is $H_r \ge 1$ such that 
$$
V(z) - V(y) \le H_r V'(x)(z-y) \quad \text{whenever} \quad 0 < x \le y \le z \le 5x \le 5r.
$$
\end{definition}
We may assume that $r \to H_r$ is nondecreasing.
It is known \cite[Section 7.1]{BGR2015} that if $\psi \in \text{WLSC}$ and $\psi \in \text{WUSC}$ then (\textbf{H}) holds.

As usual for any $x \in \R^d$ we denote by $E^x$, $P^x$ the expectation and the probability measure for the process starting from $x$. Let $D \subset \R^d$ be an open, nonempty set. By $\tau_D = \inf\{t \ge 0: \, X_t \notin D\}$ we denote the first exit time of the process $X$ from  $D$. We define a killed process $X_t^D$ by $X_t^D = X_t$ if $t < \tau_D$ and $X_t^D = \partial$ otherwise, where $\partial$ is some point adjoined to $D$. The transition density for $X_t^D$ on $D$ is given by
$$
p_D(t,x,y) = p_t(x - y) - E^x(p_{t - \tau_D}(y - X(\tau_D)), \tau_D < t), \quad t > 0, \, x,y \in D.
$$
$p_D(t,x,y)$ is called the Dirichlet heat kernel for the process $X$ on the set $D$. For any $t > 0$ we put $p_D(t,x,y) = 0$ if $x \notin D$ or $y \notin D$.
It is well known that $p_D(t,x,y) = p_D(t,y,x)$ for any $t > 0$, $x, y \in D$. The Dirichlet semigroup $\{P_t^D\}_{t > 0}$ of $X$ on an open set $D \subset \R^d$ is defined by
\begin{equation}
\label{heatsemigroup}
P_t^D f(x) = E^x(f(X_t), \, \tau_D > t), \quad x \in D,
\end{equation}
for any measurable, bounded function $f: D \to \R$. It is well known that $P_t^D f(x) = \int_D p_D(t,x,y) f(y) \, dy$, $x \in D$, $t > 0$.

The corresponding {\it{Green function}} is defined by
$$
G_D(x,y) = \int_0^{\infty} p_D(t,x,y) \, dt, \quad x,y \in D, \quad x \ne y,
$$
$G_D(x,x) = \infty$, $x \in D$, $G_D(x,y) = 0$ if $x \notin D$ or $y \notin D$.

Let $D \subset \R^d$ be a bounded, open, nonempty set. The distribution $P^x(X(\tau_D) \in \cdot)$ is called the {\it{harmonic measure}} with respect to $X$. The harmonic measure for Borel sets $A \subset (\overline{D})^c$ is given by the Ikeda-Watanabe formula \cite{IW1962}
\begin{equation}
\label{IW}
P^x(X(\tau_D) \in A) = \int_A \int_D G_D(x,y) \nu(y-z) \, dy \, dz, \quad x \in D.
\end{equation} 
When $D \subset \R^d$ is a bounded, open Lipschitz set then we have  \cite{Sztonyk2000}, \cite{Millar1975}
\begin{equation}
\label{Xboundary}
P^x(X(\tau_D) \in \partial D) = 0, \quad x \in D.
\end{equation}
It follows that for such sets $D$ the Ikeda-Watanabe formula (\ref{IW}) holds for any Borel set $A \subset D^c$.
Let $D \subset \R^d$ be a bounded, open, nonempty set. For any $s > 0$, $x \in D$, $z \in (\overline{D})^c$ put 
\begin{equation}
\label{hformula}
h_D(x,s,z) = \int_D p_D(s,x,y) \nu(y-z) \, dy.
\end{equation}
By Ikeda-Watanabe formula \cite{IW1962} and standard arguments (see e.g. \cite[proof of Proposition 2.5]{KS2006}) for any Borel sets $A \subset (0,\infty)$, $B \subset (\overline{D})^c$ we have 
\begin{equation}
\label{IW2}
P^x(\tau_D \in A, X(\tau_D) \in B) = \int_A \int_B h_D(x,s,z) \, dz \, ds, \quad x \in D.
\end{equation}
If (\ref{Xboundary}) holds then we can take $B \subset D^c$ in (\ref{IW2}). 

The main tool used in the proof of Theorem \ref{main1} is the so-called difference process already  constructed  in \cite{KR2016} (cf. also \cite{K2013}). For the Reader convenience we recall its  definition and basic properties  from \cite[Section 4]{KR2016}. We will use the following notation $\hat{x} = (-x_1,x_2,\ldots,x_d)$ for $x = (x_1,x_2,\ldots,x_d)$, $D_+ = \{(x_1,x_2,\ldots,x_d) \in D: \, x_1 > 0\}$, $D_- = \{(x_1,x_2,\ldots,x_d) \in D: \, x_1 < 0\}$ for $D \subset \R^d$. For any $t > 0$, $x,y \in \R^d_+$ put
$$
\tilde{p}_t(x,y) = p_{t}(x - y) - p_t(\hat{x} - y).
$$
Now let us define $\tilde{P}_t(x,A)$, $t \ge 0$, $x \in \R^d_+$, $A \in \calB(\R^d_+)$ by $\tilde{P}_t(x,A) = \int_A \tilde{p}_t(x,y) \, dy$, $t > 0$, and $\tilde{P}_0(x,\cdot) = \delta_x$.
Let us augment $\R^d_+$ by an extra point $\partial$ so that $\R^d_+ \cup \{\partial\}$ is a one-point compactification of $\R^d_+$. We extend $\tilde{P}_t(x,A)$ to a Markov transition function on $\R^d_+ \cup \{\partial\}$ by setting 
\begin{equation}
\label{cemetary}
\tilde{P}_t(x,A) = 
\left\{
\begin{array}{ll}
\displaystyle
\tilde{P}_t(x,A \cap \R^d_+) + 1_A(\partial) (1 - \tilde{P}_t(x,\R^d_+)), & \quad \text{for} \quad x \in \R^d_+, \\
\displaystyle
1_A(\partial), & \quad \text{for} \quad x = \partial,
\end{array}
\right.
\end{equation} 
for any $A \subset \R^d_+ \cup \{\partial\}$ which is in the $\sigma$-algebra in $\R^d_+ \cup \{\partial\}$ generated by $\calB(\R^d_+)$. Then \cite[Section 4]{KR2016} there exists a Hunt process $\tilde{X}=(\tilde{X}_t, t \ge 0)$ with the state space $\R^d_+ \cup \{\partial\}$  and the transition function $\tilde{P}_t(x,A)$. We call it the difference process. We will denote by $\tilde{P}^x$, $\tilde{E}^x$ the probability and the expected value with respect to the process $\tilde{X}$ starting from $x$. 

We say that $D \subset \R^d$ satisfies the {\it{outer cone condition}} if for any $z \in \partial D$ there exist $r > 0$ and a cone $A$ with vertex $z$ such that $A \cap B(z,r) \subset D^c$.

Let $D \subset \R^d_+$ be an open, nonempty set satisfying the outer cone condition. For any $t > 0$, $x,y \in D$ we put 
$$
\tilde{p}_D(t,x,y) = \tilde{p}_t(x,y) - \tilde{E}^x\left(\tilde{p}_{t - \tau_D}(\tilde{X}(\tau_D),y), t > \tau_D\right),
$$
where $\tau_D = \inf\{t \ge 0: \, \tilde{X}_t \notin D\}$.
For any Borel $A \subset D$, $x \in D$ and $t > 0$ we have
\begin{equation*}
\tilde{P}^x(\tilde{X}_t \in A, \tau_D > t) = \int_A \tilde{p}_D(t,x,y) \, dy.
\end{equation*}
We say that a set $D \subset \R^d$ is symmetric if for any $x \in D$ we have $\hat{x} \in D$.

Let $D \subset \R^d$ be an open, nonempty, symmetric set satisfying the outer cone condition. For any $t > 0$, $x,y \in D_+$, we have
$$
\tilde{p}_{D_+}(t,x,y) = p_D(t,x,y) - p_D(t,\hat{x},y).
$$
It follows that
\begin{equation}
\label{key}
0 \le  p_D(t, x, y) - p_D(t,\hat{x}, y)\le p_{t}(x - y) - p_t(\hat{x} - y),
\end{equation}
for any $t > 0$, $x, y \in D_+$.
We define the Green function for $\tilde{X}_t$ and $D_+$ by
$$
\tilde{G}_{D_+}(x,y) = \int_0^{\infty} \tilde{p}_{D_+}(t,x,y) \, dt, \quad x,y \in D_+,
$$
$\tilde{G}_{D_+}(x,x) = \infty$, $x \in D_+$, $\tilde{G}_{D_+}(x,y) = 0$ if $x \notin D_+$ or $y \notin D_+$.
For any $x,y \in D_+$, $x \ne y$ and a Borel set $A \subset \R^d_+$ put
$$
\tilde{\nu}(x,y) = \lim_{t \to 0} \frac{\tilde{p}_t(x,y)}{t} = \nu(x-y) - \nu(\hat{x}-y)
$$
and $\tilde{\nu}(x,A) = \int_A \tilde{\nu}(x,y) \, dy$. We call $\tilde{\nu}(x,A)$ the L{\'e}vy measure for the process $\tilde{X}$.

If $D \subset \R^d$ is a symmetric, open, nonempty, bounded Lipschitz set then for a Borel set $B \subset \R^d_+ \setminus {D}$ and $x \in D_+$ we have 
\begin{equation}
\label{IWtilde}
\tilde{P}^x\left(\tilde{X}(\tau_{D_+}) \in B \right) = 
\int_{D_+} \tilde{G}_{D_+}(x,y) \int_B \tilde{\nu}(y,z) \, dz \, dy.
\end{equation}
For any $s > 0$, $x \in D_+$, $z \in \R^d_+ \setminus \overline{D_+}$ put 
\begin{equation*}
\tilde{h}_D(x,s,z) = \int_{D_+} \tilde{p}_{D_+}(s,x,y) \tilde{\nu}(y,z) \, dy.
\end{equation*}
By (\ref{IWtilde}) and standard arguments (see e.g. \cite[proof of Proposition 2.5]{KS2006}) for any Borel sets $A \subset (0,\infty)$, 
$B \subset \R^d_+ \setminus {D_+}$ we have 
\begin{equation}
\label{IWtilde2}
\tilde{P}^x(\tau_D \in A, \tilde{X}(\tau_D) \in B) = \int_A \int_B \tilde{h}_D(x,s,z) \, dz \, ds, \quad x \in D_+.
\end{equation}

Lemma 5.2 in \cite{KR2016} gives
\begin{lemma}
\label{ABLevyquotient}
Let $X$ be a process which satisfies assumptions of Theorem \ref{main1}. Then for any $v,z \in \R^d_+$ we have
\begin{eqnarray*}
\tilde{\nu}(v,z) &\le& c |z-\hat{z}| \frac{\nu(v-z)}{1\wedge|v-z|} \left(1+\frac{|v-\hat{z}|}{|v-z|}\right),
\end{eqnarray*}
where $c=c(d,\psi)$.
\end{lemma}

For any open bounded set $D \subset \R^d$ and any $t > 0$ the operators $P_t^D$ (defined by (\ref{heatsemigroup})) acting on $L^2(D)$ are Hilbert-Schmidt operators. From general theory of semigroups it is well known that there exists a sequence of eigenvalues $0 < \lambda_1^D < \lambda_2^D \le \lambda_3^D \le \ldots$, $\lim_{n \to \infty} \lambda_n^D = \infty$ and an orthonormal basis of eigenfunctions $\{\varphi_n^D\}_{n = 1}^{\infty}$ such that
$$
P_t^D \varphi_n^D(x) = \exp(-\lambda_n^D t) \varphi_n^D(x), \quad t > 0, \, x \in D, \, n \in \N.
$$
The following lemma was proved in \cite{BGR2014a}.
\begin{lemma} \label{eigenApprox}
Let $D$ be an open bounded set containing a ball of radius $R > 0$. Then
 $$\frac 1{8} \left(\frac R  {\rm{diam} D} \right)^{2} \le  \lambda_1^DV^2(R) \le  c \left(\frac {\rm{diam} D} R\right)^{d/2},$$
where  $c=c(d)$.
\end{lemma}
For $R > 0$, by Lemma \ref{eigenApprox}, we get
\begin{equation}
\label{eigenvalue}
\lambda_1^{B(0,R)} \approx \frac{1}{V^2(R)},
\end{equation}
where the comparability constant depends only on $d$.

The following result is the partial extension of Theorem 4.5 in \cite{BGR2014a}. Recall that for any set $D\subset \R^d$ and  $x \in \R^d$,  $\delta_{D}(x)= \dist (x, D^c)$.
\begin{theorem}\label{hk_kula2}Let $R > 0$ and put $\lambda_1(R)=\lambda_1^{B(0,R)}$. If $\psi\in\WLSC{\la}{ \theta_0}{\lC}\cap\WUSC{\ua}{ \theta_0}{\uC}$, for some $\underline{\alpha} > 0$, $\overline{\alpha} \in (0,2)$, $\theta_0 \ge 0$, $\underline{C}, \overline{C}  > 0$
 and the L{\'e}vy measure has strictly positive density then we have 
 \begin{eqnarray*} 
&& c_1^{-1}
      \p^x\left(\tau_{B(0,R)}>\frac{t}{2}\right) \p^y\left(\tau_{B(0,R)}>\frac{t}{2}\right) p_{t\wedge V^2(R)}(x-y)\\
&\le& p_{B(0,R)}(t, x, y) \le c_1
      \p^x\left(\tau_{B(0,R)}>\frac{t}{2}\right) \p^y\left(\tau_{B(0,R)}>\frac{t}{2}\right) p_{t\wedge V^2(R)}(x-y)
 \end{eqnarray*}
and
$$ c_2^{-1}  e^{-\lambda_1(R) t}\left(\frac{V(\delta_{B(0,R)}(x))}{\sqrt{t}\wedge V(R)}\wedge 1\right) \le 
\p^x\left(\tau_{B(0,R)}>t\right) \le c_2  e^{-\lambda_1(R) t}\left(\frac{V(\delta_{B(0,R)}(x))}{\sqrt{t}\wedge V(R)}\wedge 1\right)
$$
for all $x,y\in B(0,R)$ and $t>0$. The constants $c_1$, $c_2$ depend on  $R$,  $d$ and $\psi$. They are nondecreasing with respect to $R$.
\end{theorem}

Let $R \in (0,1]$. An immediate corollary of Theorem \ref{hk_kula2}, (\ref{eigenvalue}) and subadditivity of $V$ is the following comparability 
\begin{equation}
\label{comp}
p_{B(0,R)}(t, x, y)\approx p_t(x-y), \quad t\le V^2(R), \, x,y \in B(0,3R/4).
\end{equation}
The comparability constant depends only on $d$ and $\psi$.

\section{Auxiliary estimates of the heat kernel and the L{\'e}vy measure}

In this section we present some estimates of $p_t(r)$ and $\nu(r)$ which will be needed in the sequel. 

The following estimate follows from \cite[Corollary 2.12]{GRT2016}.
\begin{lemma} 
\label{GRTestimates}
For any $r, t > 0$ we have
$$
p_t(r) \ge c t \nu(r) \exp\left(\frac{-c_1 t}{V^2(r)}\right),
$$
where $c = c(d)$, $c_1 = c_1(d)$.
\end{lemma}
The next lemma follows from \cite[Corollary 7]{BGR2014}.
\begin{lemma}
\label{upperptr}
For any $r, t > 0$ we have
$$
p_t(r) \le \left(p_t(0) \wedge \frac{c t}{V^2(r) r^d}\right),
$$
where $c = c(d)$.
\end{lemma}
The next two lemmas are easy consequences of the results from \cite{BGR2014} and \cite{BGR2014a}. 
\begin{lemma}
\label{p0estimates}
Let $\psi$ satisfy WLSC($\underline{\alpha},\theta_0,\underline{C}$) for some $\underline{\alpha} > 0$, $\theta_0 \ge 0$, and $\underline{C}  > 0$. Then for any $T > 0$ there exists $c = c(d,\psi,T)$ such that for any $t \in (0,T]$ we have
$$
p_t(0) \le c [V^{-1}(\sqrt{t})]^{-d}.
$$
\end{lemma}
\begin{proof}
By \cite[Lemma 1.6]{BGR2014a} there exists $T_0 = T_0(d,\psi)$ and $c_1 = c_1(d,\psi)$ such that for any $t \in (0,T_0]$ we have 
$$
p_t(0) \le c_1 [V^{-1}(\sqrt{t})]^{-d}.
$$
On the other hand for any $t \in (T_0,T]$ we have
$$
p_t(0) \le p_{T_0}(0) \le p_{T_0}(0) \frac{[V^{-1}(\sqrt{T})]^{d}}{[V^{-1}(\sqrt{t})]^{d}}.
$$
\end{proof}

\begin{lemma}
\label{nuestimates}
Let $\psi$ satisfy WLSC($\underline{\alpha},\theta_0,\underline{C}$) and WUSC($\overline{\alpha},\theta_0,\overline{C}$) for some $\underline{\alpha} > 0$, $\overline{\alpha} \in (0,2)$, $\theta_0 \ge 0$, and $\underline{C}, \overline{C}  > 0$. 
We assume also that the L{\'e}vy measure has strictly positive density.
Then for any $R_0 > 0$ there exist $c_1 = c_1(d,\psi,R_0)$, $c_2 = c_2(d)$ such that for any $r \in (0,R_0]$ we have
$$
\frac{c_1}{V^2(r) r^d}\le \nu(r) \le \frac{c_2}{V^2(r) r^d}.
$$
\end{lemma}
\begin{proof}
The upper bound follows from Lemma \ref{upperptr}.

By \cite[Corollary 22]{BGR2014} and (\ref{Vpsi}) there exists $r_0 = r_0(d,\psi)$ and $c_1 = c_1(d,\psi)$ such that for any $r \in (0,r_0]$ we have
$$
\nu(r) \ge \frac{c_1}{V^2(r) r^d}.
$$
Hence, for any $r \in (r_0,R_0]$ we obtain
$$
\nu(r) \ge \nu(R_0) \ge \nu(R_0) \frac{V^2(r_0) r_0^d}{V^2(r) r^d}.
$$
\end{proof}

\begin{remark}
Note that under the assumptions of Lemma \ref{nuestimates}, for any $R>0$,  there exists $c = c(d,\psi, R)$ such that the 
L{\'e}vy measure satisfies $\nu(r) \le c \nu(2r)$  for any $0 < r \le R$.
\end{remark}

\begin{lemma}
\label{ptestimates}
Let $\psi$ satisfy WLSC($\underline{\alpha},\theta_0,\underline{C}$) and WUSC($\overline{\alpha},\theta_0,\overline{C}$) for some $\underline{\alpha} > 0$, $\overline{\alpha} \in (0,2)$, $\theta_0 \ge 0$, and $\underline{C}, \overline{C}  > 0$. We assume also that the L{\'e}vy measure has strictly positive density. 
Fix $R > 0$. 
If $t \in (0,1\vee V^2(R)]$, $r > 0$ and $r < V^{-1}(\sqrt{t})$ then 
\begin{equation}
\label{ptestimates2}
p_t(r) \approx [V^{-1}(\sqrt{t})]^{-d},
\end{equation}
if $t > 0$, $r \in [0,R]$ and $r \ge V^{-1}(\sqrt{t})$ then 
\begin{equation}
\label{ptestimates3}
p_t(r) \approx \frac{t}{V^2(r) r^d}.
\end{equation}
For any $t \in (0,1\vee V^2(R)]$, $r \in [0,R]$ we have
\begin{equation}
\label{ptestimates1}
p_t(r) \approx \min\left\{[V^{-1}(\sqrt{t})]^{-d}, \frac{t}{V^2(r) r^d}\right\}. 
\end{equation}
The comparability constants depend only on $d$, $\psi$ and $R$.
\end{lemma}
\begin{proof}
In view of (\ref{rVd}) it is enough to show (\ref{ptestimates2}) and (\ref{ptestimates3}). 

Case 1. $r < V^{-1}(\sqrt{t})$, $t \in (0,1\vee V^2(R)]$, $r > 0$.

By the fact that $r \to p_t(r)$ is nonincreasing and Lemma \ref{GRTestimates} we get
\begin{equation}
\label{GRTestimates1}
p_t(r) \ge p_t(V^{-1}(\sqrt{t})) \ge 
c t \nu(V^{-1}(\sqrt{t})) \exp\left(\frac{-c_1 t}{V^2(V^{-1}(\sqrt{t}))}\right).
\end{equation}
Note that $V^2(V^{-1}(\sqrt{t})) = t$. Using this and Lemma \ref{nuestimates} (applied for $R_0 =R\vee  V^{-1}(1)$) we get that the right hand side of (\ref{GRTestimates1}) is bounded from below by 
$$
\frac{c t}{V^2(V^{-1}(\sqrt{t})) [V^{-1}(\sqrt{t})]^d} = c [V^{-1}(\sqrt{t})]^{-d}.
$$

On the other hand, by Lemma \ref{upperptr} and Lemma \ref{p0estimates} (applied for $T =1\vee V^2(R)$) we get
$$
p_t(r) \le p_t(0) \le c [V^{-1}(\sqrt{t})]^{-d}.
$$

Case 2. $r \ge V^{-1}(\sqrt{t})$, $t > 0$, $r \in [0,R]$.

Note that by (\ref{Vr}) we have $t/V^2(r) \le 1$. Using this, Lemma \ref{GRTestimates} and Lemma \ref{nuestimates} we get
$$
p_t(r) \ge c t \nu(r) \exp\left(\frac{-c_1 t}{V^2(r)}\right) \ge \frac{c t}{V^2(r) r^d}.
$$
By Lemma \ref{upperptr} we obtain 
$$
p_t(r) \le \frac{c t}{V^2(r) r^d}.
$$
\end{proof}

In the sequel we need estimates of $\frac{d}{dr} p_t(r)$. They are based on the following result.
\begin{theorem}\cite[Theorem 1.5]{KR2016}
\label{dertransition}
Let $X$ be a pure-jump isotropic L{\'e}vy process in $\R^d$ with the characteristic exponent $\psi$. We assume that its L{\'e}vy measure is infinite and has the density $\nu(x) = \nu(|x|)$ such that  
$\nu(r) $ is nonincreasing, absolutely continuous and $-\nu'(r)/r$ is nonincreasing.  We denote transition densities of $X$ by $p_t(x) = p_t(|x|)$. Then there exists a L{\'e}vy process $X_t^{(d+2)}$ in $\R^{d+2}$ with the characteristic exponent $\psi^{(d+2)}(\xi) = \psi(|\xi|)$, $\xi \in \R^{d+2}$ and the radial, radially nonincreasing transition density $p_{t}^{(d+2)}(x) = p_{t}^{(d+2)}(|x|)$ satisfying
\begin{equation}
\label{derivativeptr}
p_t^{(d+2)}(r) = \frac{-1}{2 \pi r} \frac{d}{dr} p_t(r), \quad \quad r > 0.
\end{equation}
Moreover,  $p_t^{(d+2)}$ is continuous at any $x\neq0$.
\end{theorem}
\begin{lemma}
\label{processplus}
Let $X$ satisfy assumptions of Theorem \ref{main1}. Then the assumptions of Theorem \ref{dertransition} are satisfied. Denote by $X_t^{(d+2)}$, $\nu^{(d+2)}(x)$, $\psi^{(d+2)}(x)$ the corresponding process in $\R^{d+2}$, the density of its L{\'e}vy measure and its characteristic exponent, respectively. Then $X_t^{(d+2)}$ is a pure-jump isotropic unimodal L{\'e}vy process in $\R^{d+2}$, its L{\'e}vy measure is infinite and satisfies
\begin{equation}
\label{Levy2plus}
\forall{R > 0} \quad \nu^{(d+2)}(R) > 0.
\end{equation}
$\psi^{(d+2)}$ satisfies WLSC($\underline{\alpha},\theta_0,\underline{C}$) and WUSC($\overline{\alpha},\theta_0,\overline{C}$).
\end{lemma}
\begin{proof}
The assumptions of Theorem \ref{dertransition}  are clearly  satisfied. The fact that $X_t^{(d+2)}$ is a pure-jump isotropic unimodal L{\'e}vy process in $\R^{d+2}$ follows directly from Theorem \ref{dertransition}. The properties of $\psi^{(d+2)}$ are also clear because $\psi^{(d+2)}(|x|) = \psi(|x|)$. The fact that the L{\'e}vy measure of $X_t^{(d+2)}$ is infinite is stated in \cite[proof of Theorem 1.5]{KR2016}. Now we will justify (\ref{Levy2plus}). Since $X_t^{(d+2)}$ is isotropic unimodal we know that $r \to \nu^{(d+2)}(r)$ is nonincreasing. By \cite[proof of Theorem 1.5]{KR2016} we have
$$
\nu^{(d+2)}(r) = \frac{-1}{2 \pi r} \frac{d \nu}{dr}(r), \quad r > 0.
$$
If $\nu^{(d+2)}(R) = 0$ for some $R > 0$ then $\nu^{(d+2)}(r) = 0$ for all $r \ge R$. But then we would have $\nu(r) = 0$ for all $r \ge R$ which contradicts assumptions in Theorem \ref{main1}.
\end{proof}
Using  Theorem \ref{dertransition} and Lemmas \ref{processplus}, \ref{ptestimates} we obtain
\begin{lemma}
\label{nablaptestimates}
Let $X$ satisfy assumptions of Theorem \ref{main1}. Fix $R > 0$. 
If $t \in (0,1]$, $r > 0$ and $r < V^{-1}(\sqrt{t})$ then 
$$
\left| \frac{d}{dr} p_t(r) \right| \approx r [V^{-1}(\sqrt{t})]^{-d-2},
$$
if $t > 0$, $r \in (0,R]$ and $r \ge V^{-1}(\sqrt{t})$ then 
$$
\left| \frac{d}{dr} p_t(r) \right| \approx \frac{t}{V^2(r) r^{d+1}}.
$$
For any $t \in (0,1]$, $r \in (0,R]$ we have
$$
\left| \frac{d}{dr} p_t(r) \right| \approx r \min\left\{[V^{-1}(\sqrt{t})]^{-d-2}, \frac{t}{V^2(r) r^{d+2}}\right\}.
$$
The comparability constants depend only on $d$, $\psi$ and $R$.
\end{lemma}

\section{Gradient estimates}

In this section we prove Theorem \ref{main1}. In the whole section we suppose that the process $X$ satisfies assumptions of this theorem.

\begin{lemma} \label{derestimates1} For any $t \in (0,1]$, $r \in (0,2]$ we have
$$
\left|\frac{d}{dr} p_t(r)\right| \le  c \left(\frac{p_t(r)}{r}\wedge \frac{p_t(r)}{V^{-1}(\sqrt{t})}\right),
$$
where $c = c(d,\psi)$.
\end{lemma}
\begin{proof}
Case 1. $r < V^{-1}(\sqrt{t})$.

By Lemmas \ref{ptestimates} and \ref{nablaptestimates} we get 
$$ \left|\frac{d}{dr} p_t(r)\right| \approx
r [V^{-1}(\sqrt{t})]^{-d-2} \le 
c [V^{-1}(\sqrt{t})]^{-d-1} \approx
\frac{p_t(r)}{V^{-1}(\sqrt{t})} =
\left(\frac{p_t(r)}{r}\wedge \frac{p_t(r)}{V^{-1}(\sqrt{t})}\right).
$$

Case 2. $r \ge V^{-1}(\sqrt{t})$.

By Lemmas \ref{ptestimates} and \ref{nablaptestimates} we get 
$$
\left|\frac{d}{dr} p_t(r)\right| \approx
\frac{t}{V^2(r) r^{d+1}} \approx 
\frac{p_t(r)}{r} =
\left(\frac{p_t(r)}{r}\wedge \frac{p_t(r)}{V^{-1}(\sqrt{t})}\right).
$$
\end{proof}

\begin{lemma} 
\label{lem1}
For any $t \in (0,1]$, $x, y \in B_+(0,1)$ we have
$$
0\le p_{t}(x - y) - p_t(\hat{x} - y) \le c  |\hat{x} - x| \left(\frac{p_{t}(x - y)}{|x-y|} \wedge \frac{p_{t}(x - y)}{V^{-1}(\sqrt{t})} \right),
$$
where $c = c(d,\psi)$.
\end{lemma}
\begin{proof}
The inequality $p_t(x-y) - p_t(\hat{x} - y) \ge 0$ is clear because $|\hat{x} - y| \ge |x-y|$.

We also have
\begin{eqnarray}
\nonumber
p_t(x-y) - p_t(\hat{x}-y) &=& p_t(|x-y|) - p_t(|\hat{x}-y|) \\
\label{xyxi}
&=& (|x-y| - |\hat{x}-y|) D p_t(|x-y| + \xi),
\end{eqnarray}
where $\xi \in (0,|\hat{x}-y| - |x-y|)$ and $Dp_t(r) = \frac{d}{dr} p_t(r)$.
By Lemma \ref{derestimates1} this is bounded from above by
$$
c  |\hat{x} - x| \left(\frac{p_{t}(|x-y| + \xi)}{|x-y| + \xi} \wedge \frac{p_{t}(|x-y| + \xi)}{V^{-1}(\sqrt{t})} \right)
\le c  |\hat{x} - x| \left(\frac{p_{t}(x - y)}{|x-y|} \wedge \frac{p_{t}(x - y)}{V^{-1}(\sqrt{t})} \right).
$$
\end{proof}

Recall that for $r > 0$, $B=B(0,r)$, $t > 0$, $x, y \in B_+$ we have
$$
\tilde{p}_{B_+}(t,x,y) = p_{B}(t,x,y)- p_{B}(t,\hat{x},y).
$$
\begin{lemma}
\label{lem4} 
Let $r \in (0,1]$, $B = B(0,r)$, $x \in B_+(0,r/16)$, $y\in B_+$.  Then for any $t \in (0,1]$ we have
$$
0 \le \tilde{p}_{B_+}(t,x,y) \le 
c |\hat{x} - x|  \left(\frac{1}{r} \vee \frac{1}{V^{-1}(\sqrt{t})}\right) p_{B}(t,x,y),
$$
where $c = c(d,\psi)$.
\end{lemma}

\begin{proof}
Let $t_0= V^2(r) \wedge \frac{t}{2}$ and  $t_1=t-t_0$. 
By the semigroup property and  (\ref{key})  
\begin{eqnarray}
\tilde{p}_{B_+}(t,x,y)&=&\int_{B_+} \tilde{p}_{B_+}(t_0,x,w)\tilde{p}_{B_+}(t_1,w,y)\, dw\nonumber\\
&\le& \int_{B_+} (p_{B}(t_0, x, w)- p_{B}(t_0, \hat{x}, w))p_{B}(t_1, w, y) \, dw \nonumber\\
&\le&\int_{B_+} (p_{t_0}(x-w)- p_{t_0}(\hat{x}-w))p_{B}(t_1, w, y) \, dw \nonumber\\
&\le& c  \frac{|\hat{x} - x|}{V^{-1}(\sqrt{t_0})} \int_{B_+} p_{t_0}(x - w) p_{B}(t_1, w, y) \, dw \label{SP}
\end{eqnarray}
where the last step follows from  Lemma \ref{lem1}.

By Theorem \ref{hk_kula2}, for $w \in B_+$, we have
$$
p_{B}(t_1, y, w)\le c\p^y\left(\tau_{B}>\frac{t_1}{2}\right)\p^w\left(\tau_{B}>\frac{t_1}{2}\right) 
p_{t_1\wedge V^2(r)}(y-w),
$$
 and since  $|x|\le r/16$ we get
$$
\p^w\left(\tau_{B}>\frac{t_1}{2}\right)\le c\p^x\left(\tau_{B}>\frac{t_1}{2}\right).
$$
Applying the last two estimates to (\ref{SP}) we obtain 
$$
\tilde{p}_{B+}(t, x, y)
\le  c \frac{|\hat{x} - x|}{V^{-1}(\sqrt{t_0})} \p^x\left(\tau_{B}>\frac{t_1}{2}\right) \p^y\left(\tau_{B}>\frac{t_1}{2}\right)   p_{t_0+(t_1 \wedge V^2(r))}(x - y).
$$
Note that  $t_0+(t_1 \wedge V^2(r))\approx  t \wedge V^2(r)$, which implies, by Lemma \ref{ptestimates}, that
$$
p_{t_0+(t_1 \wedge V^2(r))}(x - y) \approx p_{t \wedge V^2(r)}(x - y).
$$
Moreover,
$$
\p^x\left(\tau_{B}>\frac{t_1}{2}\right)\approx \p^x\left(\tau_{B}>\frac{t}{2}\right), 
\quad 
\p^y\left(\tau_{B}>\frac{t_1}{2}\right)\approx \p^y\left(\tau_{B}>\frac{t}{2}\right),  
$$
which follows from Theorem \ref{hk_kula2} and (\ref{eigenvalue}).

Finally, we infer that
$$
\tilde{p}_{B_+}(t, x, y) \le c \frac{|\hat{x} - x|}{V^{-1}(\sqrt{t_0})} p_{t \wedge V^2(r)}(x - y)\p^x\left(\tau_{B}>\frac{t}{2}\right)\p^y\left(\tau_{B}>\frac{t}{2}\right).
$$
Note that $V^{-1}(\sqrt{t_0}) = r \wedge V^{-1}(\sqrt{t/2})$. By scaling properties of $V^{-1}$ from  Lemma \ref{Vscaling} we get
$$
\frac{1}{V^{-1}(\sqrt{t_0})} = \frac{1}{r} \vee \frac{1}{V^{-1}(\sqrt{t/2})} \le
c \left( \frac{1}{r} \vee \frac{1}{V^{-1}(\sqrt{t})} \right).
$$
Observing, again by Theorem \ref{hk_kula2}, that we have
$$
p_B(t,x,y) \approx p_{t \wedge V^2(r)}(x - y) P^x\left(\tau_B > \frac{t}{2} \right) P^y\left(\tau_B > \frac{t}{2} \right)
$$
we complete the proof.
\end{proof}

\begin{lemma} 
\label{lem3}
Let $r \in (0,1]$, $B=B(0,r)$, $x\in B_+(0,r/16)$, $x = |x| e_1$, $y \in B_+(0,r/4)$. Then for every $t \in (0,1]$ we have
\begin{equation}
\label{lem3formula}
\tilde{p}_{B_+}(t,x,y) \le c \frac{|x - \hat{x}|}{|y|} p_B(t,x,y),
\end{equation}
where $c = c(d,\psi)$.
\end{lemma}

\begin{proof} 

Case 1. $|y| \le 4 |x|$.

We have
$$
\tilde{p}_{B_+}(t,x,y) \le p_{B}(t,x,y) \le  \frac{2 |x - \hat{x}|}{|y|} p_B(t,x,y).
$$

Case 2. $|y| > 4 |x|$, $t \le V^2(|x-y|)$.

Note that $|y - x| \approx |y|$. By (\ref{key}), Lemma \ref{lem1} and (\ref{comp}) we get (\ref{lem3formula}).

Case 3. $|y| > 4 |x|$, $t > V^2(|x-y|)$.

By Lemma \ref{lem4} we arrive at
$$
\tilde{p}_{B_+}(t,x,y) 
\le c |\hat{x} - x|  \left(\frac{1}{r} \vee \frac{1}{V^{-1}(\sqrt{V^2(|x-y|)})}\right) p_{B}(t,x,y)
\le c \frac{|x - \hat{x}|}{|y|} p_B(t,x,y).
$$
\end{proof}

\begin{lemma}
\label{lem5}
Let $r \in (0,1]$, $t \in (0,1]$, $B=B(0,r)$, $x\in B_+(0,r/16)$, $y\in B_+\setminus B_+(0,r/4)$. Then 
$$
\tilde{p}_{B_+}(t,x,y)
\le c \frac {|\hat{x} - x|}{r}  p_{B}(t, x, y),
$$
where $c = c(d,\psi)$.
\end{lemma}

\begin{proof}
Note that for $t \in (V^2(r) \wedge 1,1]$ we have
$$
\frac{1}{V^{-1}(\sqrt{t})} \le \frac{1}{V^{-1}(\sqrt{V^2(r)})} = \frac{1}{r}.
$$
Hence the assertion of the lemma for $t \in (V^2(r) \wedge 1,1]$ follows from Lemma \ref{lem4}. So we may assume that $t \in (0,V^2(r) \wedge 1]$.

Observe that $|x-y| \ge 3r/16$. If $y \in B_+(0,3r/4) \setminus B_+(0,r/4)$ then by (\ref{key}), Lemma \ref{lem1} and (\ref{comp}) we get
$$
\tilde{p}_{B_+}(t,x,y) \le c \frac {|\hat{x} - x|}{r}  p_{t}( x- y)\le c \frac {|\hat{x} - x|}{r}  p_{B}(t, x, y),
$$
so we may assume that $y \in B_+ \setminus B_+(0,3r/4)$.

Let $D^*=B(0,r)\setminus \overline{B(0,r/2)}$, $D_0 = \{z \in \R_+^d: \, r/4 \le |z| \le r/2\}$, $D_1= B_+(0,4|x|)$ and $D_2=B_+(0,r/4)$.
By standard arguments (the strong Markov property of $\tilde{X}$ and (\ref{IWtilde2})) we have
$$
\tilde{p}_{B_+}(t, x, y)= \int_{D^*_+}\int_0^t \tilde{p}_{D^*_+}(s, y, w)\int_{B_+\setminus D^*_+ }\tilde{\nu}(w,z)\tilde{p}_{B_+}(t-s, z, x)dz \,ds \, dw.
$$
Splitting the integration we obtain
\begin{eqnarray*}\tilde{p}_{B_+}(t, x, y)
&=& \int_{D^*_+}\int_0^t \tilde{p}_{D^*_+}(s, y, w)\int_{D_0}\tilde{\nu}(w,z)\tilde{p}_{B_+}(t-s, z, x)dz \, ds \, dw\\
&+& \int_{D^*_+}\int_0^{t} \tilde{p}_{D^*_+}(s, y, w)\int_{D_{1} }\tilde{\nu}(w,z)\tilde{p}_{B_+}(t-s, z, x) \, dz \, ds \, dw\\
&+& \int_{D^*_+}\int_0^{t} \tilde{p}_{D^*_+}(s, y, w)\int_{D_{2}\setminus D_{1} }\tilde{\nu}(w,z)\tilde{p}_{B_+}(t-s, z, x) \, dz \, ds \, dw\\
&=& \text{I} +\text{II}+ \text{III}.
\end{eqnarray*}

First we estimate $\text{I}$. By (\ref{key}) and Lemma \ref{lem1}, for $ r/4\le|z|\le r/2$, $s \in (0,t)$, we have
$$
\tilde{p}_{B_+}(t-s, z, x)\le c  | x-\hat{x}| \frac{p_{t-s}(x - z)}{|x-z|}.
$$
Since $|x - z|\approx |x - y|\approx r$,  by Lemma \ref{ptestimates} and the subadditivity of $V$, this is bounded from above by
$$
 c  | x-\hat{x}| \frac{p_{t}(x - y)}{r}.
$$
Hence, using the estimate $P^y(r/4\le |X(\tau_{D^*})|\le r/2)\le c\frac {V(\delta_B(y))} {V(r)}$ \cite[see Lemma 5.7]{KR2016}, we obtain
\begin {eqnarray*}
\text{I}&\le& c \frac{|x-\hat{x}|}{r} {p}_t ( x-y) \int_{D^*} \int_{0}^t {p}_{D^*}(s, y, w)\int_{D_0}{\nu}(w-z) \, dz \, ds \, dw\\
&\le&
c \frac{|x-\hat{x}|}{r} {p}_t ( x-y) P^y(r/4\le |X(\tau_{D^*})|\le r/2) \\
&\le&
c \frac{|x-\hat{x}|}{r} {p}_t ( x-y)\frac {V(\delta_B(y))} {V(r)}\\ 
&\le &
c \frac{|x-\hat{x}|}{r} {p}_B(t, x,y), 
\end{eqnarray*}
where the last step follows from  Theorem  \ref{hk_kula2} and (\ref{eigenvalue}). 

By Lemmas \ref{ABLevyquotient}, \ref{nuestimates} and the subadditivity of $V$ we obtain
\begin{equation}
\label{Levyest}
\tilde{\nu}(w,z)\le c |z|\frac {{\nu}(w-z)}{r} \approx |z|\frac {{\nu}(x-y)}{r},  
\end{equation}
for $w\in D_+^*$ and $z\in D_{2}$.
In particular,  $\tilde{\nu}(w,z)\le \frac {c}{r} |x-\hat{x}|{\nu}(w-z)$, for $w\in  D_+^*$ and $z\in D_{1}$. Hence,  we get
\begin {eqnarray*}
\text{II}&\le& 
c \frac {|x-\hat{x}|}{r} \int_{D^*} \int_{0}^t {p}_{D^*}(s, y, w)\int_{ D_{1}} \nu(w-z){p}_{B}(t-s, z, x)\, dz \, ds \, dw \\
&\le& c \frac {|x-\hat{x}|}{r} {p}_{B}(t, y, x).
\end{eqnarray*}
Next, using (\ref{key}) and Lemma \ref{lem1}, for $s \in (0,t)$, $z\in D_{2}\setminus D_{1}$ we get
\begin{equation}
\label{tildepDt}
\tilde{p}_{B_+}(t-s, z, x)\le c  \frac{| x-\hat{x}|}{|x-z|} p_{t-s}(x - z)
\le c \frac{| x-\hat{x}|}{|z|} p_{t-s}(x - z).
\end{equation}
Moreover, since $0 < t \le V^2(r)$,  by (\ref{comp}),  for $s \in (0,t)$, $z\in D_{2}\setminus D_{1}$ we get
\begin{equation}
\label{ptscomp}
p_{t-s}(x - z)\approx {p}_{B}(t-s, z, x).
\end{equation}
Combining  (\ref{Levyest}), (\ref{tildepDt}) and  (\ref{ptscomp}) we obtain 
\begin {eqnarray*}
\text{III}&\le&
c  \frac{|x-\hat{x}|}{r} \int_{D^*_+}\int_{0}^t \tilde{p}_{D^*_+}(s, y, w)\int_{ D_{2} \setminus D_1}{\nu}(w-z){p}_{B}(t-s, z, x)\, dz \,ds \, dw\\&\le&
c  \frac{|x-\hat{x}|}{r} \int_{D^*} \int_{0}^t {p}_{D^*}(s, y, w)\int_{ D_{2} \setminus D_1} {\nu}(w-z){p}_{B}(t-s, z, x) \, dz \, ds \, dw \\
&\le& c  \frac{|x-\hat{x}|}{r} {p}_{B}(t, y, x).
\end{eqnarray*}
The proof is completed.
\end{proof}

\begin{proposition}\label{m_estimate}
Let $D$ be an open set, $0\in D$ and $r=\delta_D(0)\wedge1$. Let $x=|x|e_1, \ |x| < r/16$ and $y\in D$. For any $t \in (0,1]$ we have 
\begin{equation}
\label{pDestimate}
|{p}_{D}(t, x, y)- {p}_{D}(t, \hat{x}, y) | \le c |\hat{x} - x| \left[\frac{1}{r} \vee \frac{1}{V^{-1}(\sqrt{t})}\right] p_{D}(t, x, y),
\end{equation}
where $c = c(d,\psi)$.
\end{proposition}

\begin{proof}
We set $B=B(0,r)$ and we put
$$
k_D(t, x, y)= \int_{ B}\int_0^t p_{B}(s, x, w)\int_{D\setminus B } \nu(w-z) p_{D}(t-s, z, y) \, dz \, ds \, dw.
$$  
By standard arguments (the strong Markov property and (\ref{IW2}))
$$
p_{D}(t, x, y)= p_B(t, x, y)+ k_D(t, x, y).
$$
Lemma \ref{lem4} yields the estimate 
\begin{equation}|{p}_{B}(t, x, y) - {p}_{B}(t, \hat{x}, y)|\le c |\hat{x} - x| \left[\frac{1}{r} \vee \frac{1}{V^{-1}(\sqrt{t})}\right] p_{D}(t, x, y).\label{delta}\end{equation}

Next, we  estimate $|k_D(t, x, y) - k_D(t, \hat{x}, y)|$. For $s \in (0,t)$, $w \in B$, let 
$$
g_s(w)= \int_{D\setminus B }{\nu}(w-z){p}_{D}(t-s, z, y) \, dz.
$$
Note that for $w \in B_+(0, r/4)$ and $z\in D\setminus B$ we have, due to Lemma \ref{ABLevyquotient},
$$
|\nu(w-z)-\nu(\hat{w}-z)|\le c \frac{|w|}{r} \nu(w-z).
$$
Hence, for $s \in (0,t)$, $w \in B_+(0, r/4)$, 
\begin{eqnarray}
\nonumber
|g_s(w)-g_s(\hat{w})| &\le& \int_{D\setminus B }|{\nu}(w-z)-{\nu}(\hat{w}-z)|{p}_{D}(t-s, z, y) \, dz\\
\nonumber
&\le & c \frac{|w|}{r} \int_{D\setminus B }{\nu}(w-z){p}_{D}(t-s, z, y) \, dz\\
\label{g_estimate3}
&= & c \frac{|w|}{r} g_s(w).
\end{eqnarray}
By simple manipulations we obtain
\begin{eqnarray*}
k_D(t, x, y)-k_D(t, \hat{x}, y)&=& \int_{ B_+}\int_0^t( {p}_{B}(s, x, w) -{p}_{B}(s, \hat{x}, w))(g_s(w)-g_s(\hat{w})) \, ds \, dw\\
&=& \int_{ B_+(0, r/4)}\int_0^t +  \int_{B_+\setminus B_+(0, r/4)}\int_0^t\\
&=&  \text{I} + \text{II}. 
\end{eqnarray*}

By Lemma \ref{lem3} and (\ref{g_estimate3}) we get 
\begin{eqnarray*}
|\text{I}|&\le& c \int_{  B_+(0, r/4)}\int_0^t( {p}_{B}(s, x, w) -{p}_{B}(s, \hat{x}, w))\frac{|w|}{r}g_s(w) \, ds \, dw\\
&\le & c \frac{|x-\hat{x}|}{r} \int_{ B(0, r/4)} \int_0^t p_{B}(s,x,w) g_s(w)\, ds \, dw\\ 
&\le&  c \frac{|x-\hat{x}|}{r} k_D(t, x, y).
\end{eqnarray*}

Note that by Theorem \ref{hk_kula2}, the subadditivity of $V$ and Lemma \ref{ptestimates}, for $s \in (0,t)$ and $w \in B \setminus B(0,r/4)$, we have
$$
p_B(s,x,w) \approx p_B(s,0,w).
$$
Using this and Lemma \ref{lem5} we get 
\begin{eqnarray*}|\text{II}|&\le& 
c \frac{|x-\hat{x}|}{r} \int_{B}\int_0^t p_{B}(s, 0, w) (g_s(w)+g_s(\hat{w})) \, ds \, dw\\
&=& 2c \frac{|x-\hat{x}|}{r} \int_{B}\int_0^t p_{B}(s, 0, w) g_s(w) \, ds \, dw\\
&\approx&  \frac{|x-\hat{x}|}{r} \int_{B}\int_0^t p_{B}(s, x, w) g_s(w) \, ds \, dw\\
&=& \frac{|x-\hat{x}|}{r} k_D(t,x,y).
\end{eqnarray*}
Hence, 
 $$|k_D(t, x, y) - k_D(t, \hat{x}, y)|\le c\frac{|x-\hat{x}|}{r} k_D(t,x,y)$$
which combined with the estimate (\ref{delta}) completes the proof. 

\end{proof}

\begin{lemma}
\label{existence}
Let $D \subset \R^d$ be an open set. $\nabla_x p_D(t,x,y)$ is well defined for any $t > 0$, $x,y \in D$. 
\end{lemma}
\begin{proof}
Recall that 
$$
p_D(t,x,y) = p_t(x-y) - E^y(p_{t-\tau_D}(X(\tau_D)-x), \tau_D < t), \quad t > 0, \, x,y \in D.
$$
Since $\psi$ satisfies the Hartman-Wintner condition it is well known that for each $t > 0$ the function $x \to p_t(x)$ has derivatives of all orders on $\R^d$ \cite[Lemma 3.1]{KS2013}. By Theorem \ref{dertransition} and Lemma \ref{upperptr} for any $s \in (0,t)$, $x \in D$, $\nabla_x p_s(x-X(\tau_D))$ is well defined and
$$
|\nabla_x p_s(x-X(\tau_D))| \le \frac{c t}{V^2(\delta_D(x)) \delta_D^{d+1}(x)},
$$
where $c = c(d,\psi)$. 

So by the bounded convergence theorem $\nabla_x E^y(p_{t-\tau_D}(X(\tau_D)-x), \tau_D < t)$
is well defined for any $t > 0$, $x,y \in D$.
\end{proof}

\begin{proof}[proof of Theorem \ref{main1}]
The existence of $\nabla_x p_D(t,x,y)$ follows from Lemma \ref{existence}. 

Denote $D_i p_D(t,x,y) = \frac{\partial}{\partial x_i} p_D(t,x,y)$. Choose $z, y \in D$ and put $r = \delta_D(z) \wedge 1$. We will estimate $D_1 p_D(t,z,y)$. Estimates for $D_i p_D(t,z,y)$, $i \ne 1$ may be obtained in the same way. We may assume that $z = 0$. Choose $\varepsilon \in (0,r/16)$. Putting $x = h e_1$, ($h \in (0,\varepsilon)$) in (\ref{pDestimate}) we obtain
$$
\sup_{0 < h < \varepsilon} \left|\frac{p_D(t,he_1,y) - p_D(t,-he_1,y)}{2h}\right| \le 
c \left[\frac{1}{r} \vee \frac{1}{V^{-1}(\sqrt{t})}\right] \sup_{0 < h < \varepsilon} p_{D}(t, h e_1, y),
$$
which implies
$$
\left| D_1 p_D(t,z,y) \right| \le c \left[\frac{1}{r} \vee \frac{1}{V^{-1}(\sqrt{t})}\right]  p_{D}(t, z, y).
$$
Finally using (\ref{Vpsi-1}) we obtain (\ref{maineq}).
\end{proof}

\section{Examples}

\begin{example}
\label{stable}
Let $X$ be the isotropic $\alpha$-stable process in $\R^d$, $\alpha \in (0,2)$, $d \in \N$. We have $\psi(x) = |x|^{\alpha}$ and $\nu(x) = C_{d,\alpha} |x|^{-d-\alpha}$, where $C_{d,\alpha} = \frac{2^{\alpha} \Gamma((d+\alpha)/2)}{\pi^{d/2}|\Gamma(-\alpha/2)|}$. It is clear that $X$ satisfies assumptions of Theorem \ref{main1}. Hence, for any open, nonempty set $D \subset \R^d$ we have
\begin{equation}
\label{maineq1}
|\nabla_x \, p_D(t,x,y)| \le \frac {c}{\delta_D(x)\wedge t^{1/\alpha}} \,\, p_{D}(t, x, y), \quad \quad x, y \in D, \,\,  t \in (0,1],
\end{equation}
where $c = c(d,\alpha)$.
\end{example}

\begin{example}
\label{relativistic}
Let $X$ be the  the relativistic process in $\R^d$, (see e.g. \cite{C1989}, \cite{R2002}). We have $\psi(x) = \sqrt{|x|^2 + m^2} - m$, $m > 0$, 
$$
\nu(x) = 2^{\frac{1-d}{2}} \pi^{\frac{-d-1}{2}} m^{\frac{d+1}{2}} |x|^{\frac{-d-1}{2}} K_{\frac{d+1}{2}}(m|x|),
$$ 
where $K_{s}(r)$, $s \in \R$, is the modified Bessel function of the second kind with index $s$ (called also Macdonald function), given by 
$$
K_s(r) = 2^{-1-s} r^s \int_0^{\infty} e^{-u} e^{-r^2/(4u)} u^{-1-s} \, du, \quad r > 0.
$$
The generator of this process $m - \sqrt{m^2 - \Delta}$ is called the relativistic Hamiltonian and it is used in some models of mathematical physics (see e.g. \cite{LS2010}). One can check that  $X$ satisfies assumptions of Theorem \ref{main1}. Hence, for any open, nonempty set $D \subset \R^d$ we have
\begin{equation}
\label{maineq2}
|\nabla_x \, p_D(t,x,y)| \le \frac {c}{\delta_D(x)\wedge t} \,\, p_{D}(t, x, y), \quad \quad x, y \in D, \,\,  t \in (0,1],
\end{equation}
where $c = c(d,m)$.
\end{example}
 
\begin{example}
\label{subordinate} Let $X_t = B_{S_t}$ where $B$ is the Brownian motion in $\R^d$ (with a generator $\Delta$) and $S$ is an independent subordinator with the Laplace exponent $\phi$.
 We assume that the Levy measure of the subordinator $S$  is infinite, $\phi$ is a complete Bernstein function and it satisfies
\begin{equation}
\label{phiscaling}
c_1 \lambda^{\alpha/2} \ell(\lambda) \le \phi(\lambda) \le c_2 \lambda^{\alpha/2} \ell(\lambda), \quad \quad \lambda \ge 1,
\end{equation}
where $0 < \alpha < 2$, $\ell$ varies slowly at infinity, i.e. $\forall x > 0$ $\lim_{\lambda \to \infty} \frac{\ell(\lambda x)}{\ell(\lambda)} = 1$. (Clearly processes from Examples \ref{stable}, \ref{relativistic} satisfies these assumptions).

We have $\psi(x) = \phi(|x|^2)$. By (\ref{phiscaling}) $\psi$ satisfies WLSC($\underline{\alpha},\theta_0,\underline{C}$) and WUSC($\overline{\alpha},\theta_0,\overline{C}$) for some $\underline{\alpha} > 0$, $\overline{\alpha} \in (0,2)$, $\theta_0 \ge 0$, and $\underline{C}, \overline{C}  > 0$. The assumptions concerning the L{\'e}vy measure in Theorem \ref{main1} are satisfied by \cite[Proposition 1.3 and the proof of Example 7.1]{KR2016}.

Hence, for any open, nonempty set $D \subset \R^d$ we have
\begin{equation}
\label{maineq3}
|\nabla_x \, p_D(t,x,y)| \le c \left[\frac {1}{\delta_D(x)\wedge1} \vee \sqrt{\phi^{-1}(1/t)}\right] p_{D}(t, x, y), \quad \quad x, y \in D, \,\,  t \in (0,1],
\end{equation} 
where $c = c(d,\phi)$.
\end{example}

The process in the next example is not a subordinate Brownian motion cf. \cite[Example 7.4]{KR2016}.
\begin{example} 
\label{notsubordinate}
Let $\{X_t\}$ be the pure-jump isotropic L{\'e}vy process in $\R^d$ with the L{\'e}vy measure $\nu(dx) = \nu(|x|) \, dx$ given by the formula
\[ \nu(r) = \left\{              
\begin{array}{ll}  \mathcal{A}_{d,\alpha} r^{-d-\alpha}& \text{for} \quad r \in (0,1]\\
c_1 e^{-c_2 r} & \text{for} \quad r \in (1,\infty)          \end{array}       
\right. \]
where $\mathcal{A}_{d,\alpha} r^{-d-\alpha}$ is the L{\'e}vy density  for the isotropic $\alpha$-stable process in $\R^d$, $\alpha \in (0,2)$, $d \in \N$ and $c_1 = \mathcal{A}_{d,\alpha} e^{d+\alpha} > 0$, $c_2= d + \alpha >0$ are chosen so that $\nu(r) \in C^1(0,\infty)$. 

Note that $\psi(x) = \int_{\R^d} (1 - \cos\langle x,y \rangle) \, \nu(dy)$ behaves for $|x| \ge 1$ like the characteristic exponent for the isotropic $\alpha$-stable process so it satisfies 
\begin{equation*}
c_1 |x|^{\alpha} \le \psi(x) \le c_2 |x|^{\alpha}, \quad |x| \ge 1,
\end{equation*}
where $c_1 = c_1(d,\alpha)$, $c_2 = c_2(d,\alpha)$.  
Hence, $\psi \in \WLSC{\alpha}{1}{\underline{C}}) \cap \WUSC{\alpha}{1}{\overline{C}}$ for some $\underline{C}, \overline{C}  > 0$.

The assumptions concerning the L{\'e}vy measure in Theorem \ref{main1} are easy to check.

Note that we have $\psi^{-}(x) \approx |x|^{1/\alpha}$ for $|x| \ge 1$, where the comparability constant depends only on $d$ and $\alpha$. Hence, for any open, nonempty set $D \subset \R^d$ we have
\begin{equation}
\label{maineq4}
|\nabla_x \, p_D(t,x,y)| \le \frac {c}{\delta_D(x)\wedge t^{1/\alpha}} \,\, p_{D}(t, x, y), \quad \quad x, y \in D, \,\,  t \in (0,1],
\end{equation}
where $c = c(d,\alpha)$.

\end{example}

\section{Appendix}
The section is devoted to the proof of Theorem \ref{hk_kula2}. A similar result for smooth bounded domains was proved 
  in \cite [Theorem 4.5] {BGR2014a}, but the dependence of constants therein seems to be unclear and one can not infer  uniform estimates of the Dirichlet heat kernels as in Theorem \ref{hk_kula2}.  We follow   the arguments from the proof \cite [Theorem 4.5] {BGR2014a}, but we pay more attention to the behaviour of the constants.  To make the exposition self-contained we need to introduce some notation and to cite several results obtained in \cite{BGR2015, BGR2014a}.
	For $R > 0$ we denote $B_R = B(0,R)$.
Let $\underline{\alpha}>0$.  For $R>0$ we introduce the following quantities:  
\begin{eqnarray*} 
\underline{C}_R&=& \inf_{y\ge x\ge \frac1R }\frac{\psi(y)}{\psi(x)}\left(\frac{x}{y}\right)^{\underline{\alpha}};\\
\tilde{C}_R &=& \inf_{0<t\le  V^2(R), |x|\le R}  \frac{p_t(x)}{p_{t/2}(0)\wedge\frac{t} { V^2(|x|)|x|^d}};\\
C_R &=& 1\wedge\inf_{0<t\le V^2(|x|), |x|\le R}  \frac{p_t(x)}{t} { V^2(|x|)|x|^d};\\
 C^*_R &=& \inf_{ |x|\le R} \nu(x) { V^2(|x|)|x|^d};\\
\mathcal{I}_R&=&\inf_{ 0<\rho\leq R/2}\nu(B_{R}\setminus B_\rho) V^2(\rho).
\end{eqnarray*}

\begin{remark} At first we observe that if there exists $R > 0$ such that $\underline{C}_R>0$ then we get $\psi \in \text{WLSC}(\underline{\alpha},\frac1R,\underline{C_R})$. Consequently,  due to \cite[Lemma 12]{BGR2014}, $\underline{C}_R>0$ for any $R>0$. On the other hand, if there exists $R > 0$ and $\lC > 0$ such that $\psi \in \text{WLSC}(\underline{\alpha},\frac1R,\lC)$ then we get $\underline{C}_R \ge \lC$.
\end{remark}

\begin{remark}
If $R > 0$, $\psi\in\WLSC{\la}{ \theta_0}{\lC}\cap\WUSC{\ua}{ \theta_0}{\uC}$, for some $\underline{\alpha} > 0$, $\overline{\alpha} \in (0,2)$, $\theta_0 \ge 0$, $\underline{C}, \overline{C}  > 0$ and the L{\'e}vy measure has strictly positive density then the constants $\tilde{C}_R$, $C_R$, $C^*_R$, $\mathcal{I}_R$ are strictly positive. Indeed, by Lemma \ref{nuestimates} we get $C^*_R > 0$. By Lemma \ref{ptestimates} we get $\tilde{C}_R \wedge C_R > 0$.  Moreover, by elementary calculations    $\mathcal{I}_R\ge c(d) C^*_R$ and $C_R\ge c(d)\tilde{C}_R$.
\end{remark}

\begin{lemma}\cite[Lemma 1.6]{BGR2014a}\label{sup_p_t}
If $\psi\in \WLSC{\la}{\theta_0}{\lC}$, $r>0$ and
$0<t\le rV^2(1/\theta_0)$, then
\begin{equation*}\label{sup-p-t-1}
c_2e^{-c_1 r}\left[V^{-1}\left( \sqrt{t/r}\right)\right]^{{-}d}\leq p_t(0) \le c_3\left( 1+ (\lC r)^{-1-d/\la}\right)
\left[V^{-1}\left( \sqrt{t/r}\right)\right]^{{-}d},
\end{equation*}
where $c_1$ is an absolute constant, $c_2=c_2(d)$ and  $c_3=c_3(d,\la)$.
\end{lemma}

\begin{corollary} \label{p_t0} Let $R >0$ and $0<r\le 1$.  Under the assumptions of the previous lemma there are $c_1=c_1(d)$ and  $c_2=c_2(d,\la)$ such that 

\begin{equation*}
c_1\frac 1{R^d}\leq p_{rV^2(R)}(0) \le c_2 \frac 1 {(\lC_R r)^{1+d/\la}}\frac 1{R^d}.
\end{equation*}
\end{corollary}
\begin{proof} If  $\psi\in \WLSC{\la}{\theta_0}{\lC}$ then $\lC_R>0$ and  $\psi\in \WLSC{\la}{1/R}{\lC_R}$. Therefore, the conclusion follows from Lemma \ref{sup_p_t} with $\theta_0=1/R, \lC=\lC_R\le 1$. 
\end{proof}

\begin{lemma}\cite[Proposition 6.1]{BGR2015}\label{L5a1}
Let  the condition (\textbf{H}) hold. There are $ c_1=c_1(d)<1$ and $c_2=c_2(d)$ such that for $R>0$,
\begin{eqnarray*}\p^x(\tau_{B_{R}}>t)&\ge&   c_2\,\frac{\mathcal{I}_R}{H_R}\left(\frac{V(\delta_{B_{R}}(x))}{\sqrt{t}}\wedge1\right), \qquad
0<t\le c_1 V^2(R), \quad x\in \R^d.
\end{eqnarray*}
\end{lemma}

		\begin{lemma}\cite[Corollary 2.8]{BGR2014a}\label{hk_kula1}
Let
$D$ be  open and convex.
			 If $\psi\in\WLSC{\la}{\theta_0}{ \lC}$, $t>0$, $|x-y|<  1/\theta_0$ and $x,y\in D$, then    there  is a constant $C=C(d,\la)$ such that for all $t>0$,
			\begin{eqnarray*}
			p_{D}(t, x, y)
    &\le&  \frac{C} {\lC^{2(1+d)/\la+1}}  \left(\frac{V(\delta_D(x))}{\sqrt{t} }\wedge1 \right) \left(\frac{V(\delta_D(y))}{\sqrt{t}}\wedge1 \right) \\
    && \times \left(p_{t/2}(0) \wedge \frac {t}{V^2(|x- y|) |x- y|^d}\right).
    \end{eqnarray*}
		\end{lemma}
		
		\begin{lemma} \cite[Lemma 4.2]{BGR2014a}\label{UB10}Let $D$ be a bounded open set  and $t_0>0$. For $t\ge t_0$ and $x,y\in D$,
 \begin{eqnarray*}p_D\left(t,x,y\right) \le |D|\,(p_{t_0/4}(0))^2 \, \p^x\left(\tau_{D}>\frac{t_0}4\right) \p^y\left(\tau_{D}>\frac{t_0}4\right)  e^{\lambda_1t_0}e^{-\lambda_1t} \end{eqnarray*}
where $\lambda_1= \lambda_1^D$ and $|D|$ is the volume of $D$. 
  \end{lemma}
	
	To get the uniform  lower  bound of the heat kernel we use the following result, which is a direct consequence of  \cite[Theorem 3.3]{BGR2014a}.

\begin{lemma} \label{lower_hk_estimate1}
Let $R>0$. Assume that $\psi\in\WLSC{\la}{ \theta_0}{\lC}\cap\WUSC{\ua}{ \theta_0}{\uC}$, for some $\underline{\alpha} > 0$, $\overline{\alpha} \in (0,2)$, $\theta_0 \ge 0$, $\underline{C}, \overline{C}  > 0$
 and the L{\'e}vy measure has strictly positive density.
 Then there exist $c=c(d)<1, c_1=c_1(d,\la)$
 such  that
	\begin{eqnarray*} 
	p_{B_R}(t,x,y) &\ge&   \frac{ c_1 (\underline{C}_R)^{1+d/\la}(C_R)^{{{9}+d/\la}}}{H_R^2}  \left(\frac{V(\delta_{B_R}(x))}{\sqrt{t}}\wedge1\right)\left(\frac{V(\delta_{B_R}(y))}{\sqrt{t}}\wedge1\right)\\
		&& \times \big( p_{t/2}(0)\wedge [t\,\nu(2|x-y|)]\big),	
		\end{eqnarray*}
provided  $0<t\le  c V^2(R)C_R $ and  $x,y \in B_R$.
\end{lemma}

To deal with  the lower bound for large $t$ we use the following result. 		
	\begin{lemma}\cite [Lemma 4.3]{BGR2014a} \label{LB10}
Let D be a bounded open set.  If $t_0>0$ and
$c_*>0$ are such that
\begin{equation} \label{lower_100}
p_{D}\left(\frac{t_0}{2},x,y\right)
\ge c_*   \p^x\left(\tau_{{D}}>\frac{t_0}{2}\right) \p^y\left(\tau_{{D}}>\frac{t_0}{2}\right), \quad x, y \in {D},
\end{equation}
then for $t\ge t_0$ and $x,y\in {D}$,
$$p_{D}(t,x,y)\ge \left(\frac{c_*}{\sqrt{|{D}|}p_{t_0/2}(0)}\right)^2e^{-\lambda_1t_0} \p^x\left(\tau_{{D}}>\frac{t_0}{2}\right) \p^y\left(\tau_{{D}}>\frac{t_0}{2}\right)
  e^{-\lambda_1t}.$$
  \end{lemma}
	
	Now, we are in a position to prove Theorem \ref{hk_kula2}. 
		
		\begin{proof} [Proof of Theorem \ref{hk_kula2}]
		Fix $R > 0$. In the whole proof we understand that all inequalities hold for all $x, y \in B_R$.
		Observe that $\psi \in \WLSC{\la}{\frac1{2R}}{\underline{C}_{2R}}$. Hence, using Lemma \ref{hk_kula1}
		for $D=B_R$,  we find a constant $c_1=c_1(d, \la)$ such that for any $t>0$,
		\begin{eqnarray}
		\nonumber
		p_{B_R}(t, x, y)
    &\le&  \frac{c_1} {\underline{C}_{2R}^{2(1+d)/\la+1}}  \left(\frac{V(\delta_{B_R}(x))}{\sqrt{t} }\wedge1 \right) \left(\frac{V(\delta_{B_R}(y))}{\sqrt{t}}\wedge1 \right) \\
   \label{UB100}
    && \times \left(p_{t/2}(0) \wedge \frac {t}{V^2(|x- y|) |x- y|^d}\right).
    \end{eqnarray}
		Let  $t_0= V^2(R)$, $\lambda_1 = \lambda_1^{B(0,R)}$. By Lemma \ref{eigenApprox},  $\lambda_1t_0\le c(d)$. Moreover, for $t\le t_0$ we have the estimate 
		$$\left(p_{t/2}(0) \wedge \frac {t}{V^2(|x- y|) |x- y|^d}\right)\le \frac 1{\tilde{C}_{2R}} p_t( x-y).$$ Consequently, applying  (\ref{UB100}), for $t\le t_0$ we obtain 
		
		\begin{eqnarray}
		\nonumber
		p_{B_R}(t, x, y)
    &\le&  \frac{c_2} {\tilde{C}_{2R}\underline{C}_{2R}^{2(1+d)/\la+1}}  \left(\frac{V(\delta_{B_R}(x))}{\sqrt{t} }\wedge1 \right) \left(\frac{V(\delta_{B_R}(y))}{\sqrt{t}}\wedge1 \right) \\
   \label{UB105}
    && \times p_t( x-y)e^{-\lambda_1t}
    \end{eqnarray}
with $c_2 =  c_2( d, \la)$.
		

	Next, we deal with $t\ge t_0$. 
	 By   Corollary  \ref{p_t0} with  $r = 1/4$, 
	$ p_{t_0/4}(0)\le \frac{c_3}{ \underline{C}^{1+d/\la}_{R}}  \frac 1{R^d}$ where $c_3 =  c_3( d, \la)$. Moreover, subadditivity of $V$  yields  $\frac 1{R^d}\le\frac { 2^{d+2}}{C_{2R}} p_{V^2(R)}( 2R) \le \frac { 2^{d+2}}{C_{2R}}p_{V^2(R)}(x-y)$.
	 Applying  Lemma \ref{UB10} for $D=B_R$ with $t_0= V^2(R)$, we obtain for  $t\ge t_0$, 
$$	p_{{B_R}}\left(t,x,y\right) \le \frac{c_4}{C_{2R} \underline{C}^{2(1+d/\la)}_{R}} p_{V^2(R)}( x-y) \, \p^x\left(\tau_{D}>\frac{t_0}4\right) \p^y\left(\tau_{D}>\frac{t_0}4\right) e^{-\lambda_1t}$$
	where $c_4 =  c_4( d, \la)$. 
	 We also have (see \cite[Theorem 3.1]{KMR2013}) 
\begin{equation}\label{tau_estimate}\p^x(\tau_{B_R}>t)\le c_5\left(\frac{V(\delta_{B_R}(x))}{\sqrt{t}}\wedge1\right)\end{equation}
for an absolute constant  $c_5$. 
		Hence,   for $t\ge t_0 = V^2(R)$, we arrive at 
	\begin{equation}\label{UB101}	p_{{B_R}}\left(t,x,y\right) \le \frac{c_6}{C_{2R} \underline{C}^{2(1+d/\la)}_{R}} p_{V^2(R)}( x-y) \, \left(\frac{V(\delta_{B_R}(x))}{\sqrt{t}\wedge V(R)}\wedge1\right) \left(\frac{V(\delta_{B_R}(y))}{\sqrt{t}\wedge V(R)}\wedge1\right) e^{-\lambda_1t}\end{equation}
with  $c_6 =  c_6( d, \la)$.

Therefore, by  (\ref{UB105}) and (\ref{UB101}), we can find ${\mathcal A}_R = \frac{c_6}{C_{2R} \underline{C}^{2(1+d/\la)}_{R}}+  \frac{c_2} {\tilde{C}_{2R}\underline{C}_{2R}^{2(1+d)/\la+1}} $ such that   for all $t>0$ 
	$$	p_{{B_R}}\left(t,x,y\right) \le{\mathcal A}_R p_{t\wedge V^2(R)}(x- y) \, \left(\frac{V(\delta_{B_R}(x))}{\sqrt{t}\wedge V(R)}\wedge1\right) \left(\frac{V(\delta_{B_R}(y))}{\sqrt{t}\wedge V(R)}\wedge1\right) e^{-\lambda_1t},$$
	which is the desired uniform upper bound, since finite  ${\mathcal A}_R$ is nondecreasing with $R$.
	Integrating the above bound with respect to $y$ over  $B_R$ we obtain
	
	$$\p^x(\tau_{B_R}>t)\le {\mathcal A}_R  \left(\frac{V(\delta_{B_R}(x))}{\sqrt{t}\wedge V(R)}\wedge1\right) e^{-\lambda_1t}.$$

Now, we  deal with the lower bound. By Lemma \ref{lower_hk_estimate1}	there are $c_7=c_7(d)<1, c_8=c_8(d,\la) $ such that 

\begin{eqnarray*}
p_{{B_R}}(t,x,y) &\ge&  \frac{ c_8 (\underline{C}_R)^{1+d/\la}(C_R)^{{{9}+d/\la}}}{H_R^2}    \left(\frac{V(\delta_{B_R}(x))}{\sqrt{t}}\wedge1\right)\left(\frac{V(\delta_{B_R}(y))}{\sqrt{t}}\wedge1\right) \\
&& \times		\big( p_{t/2}(0)\wedge t\,\nu(2|x-y|)\big),
\end{eqnarray*}
provided  $0<t\le  c_7 V^2(R)C_R $. 

Next, using subadditivity  of $V$, we observe that  
 $$\nu(2|x-y|) \ge  \frac {C_{4R}^*}{2^{d+2} V^2(|x-y|)|x-y|^d}.$$
 Therefore, by the estimate (see Lemma \ref{upperptr}),  
$$p_{t/2}(0)\wedge \frac t{ V^2(|x-y|)|x-y|^d}\ge c_9 {p}_t(x-y) $$
with $c_9=c_9(d) $, we obtain

\begin{eqnarray} 
\nonumber
p_{{B_R}}(t,x,y) &\ge& \frac{ c_8 (\underline{C}_R)^{1+d/\la}(C_R)^{{{9}+d/\la}}}{H_R^2}c_9 
\left(1\wedge\frac {C_{4R}^*}{2^{d+2}}\right)  \left(\frac{V(\delta_{B_R}(x))}{\sqrt{t}}\wedge1\right)\\
&& \times \left(\frac{V(\delta_{B_R}(y))}{\sqrt{t}}\wedge1\right) p_{t}(x-y),
\label{LB100}
\end{eqnarray} 
provided  $0<t\le  c_7 V^2(R)C_R$. Applying (\ref{tau_estimate}) 
we have 
$$ p_{{B_R}}(t,x,y)\ge  \frac{ c_8 (\underline{C}_R)^{1+d/\la}(C_R)^{{{9}+d/\la}}}{(c_5\,H_R)^2}c_9 
\left(1\wedge\frac {C_{4R}^*}{2^{d+2}}\right)    \p^x(\tau_{{B_R}}>t) \p^y(\tau_{{B_R}}>t)
		 p_{t}(x-y).$$
In particular, taking $t_0= c_7 V^2(R)C_R\le V^2(R)$, we have 
\begin{eqnarray*}
p_{{B_R}}(t_0/2,x,y) &\ge&  \frac{ c_8 (\underline{C}_R)^{1+d/\la}(C_R)^{{{9}+d/\la}}}{(c_5\,H_R)^2}c_9 
\left(1\wedge\frac {C_{4R}^*}{2^{d+2}}\right)  p_{t_0/2}(2R)\\
&& \times   \p^x(\tau_{{B_R}}>t_0/2) \p^y(\tau_{{B_R}}>t_0/2).
\end{eqnarray*}
	To extend the estimate (\ref{LB100}) to $t\ge t_0$ we apply  Lemma \ref{LB10} with 
	$$c^*= \frac{ c_8 (\underline{C}_R)^{1+d/\la}(C_R)^{{{9}+d/\la}}}{(c_5\,H_R)^2}c_9
\left(1\wedge\frac {C_{4R}^*}{2^{d+2}}\right)  p_{t_0/2}(2R).$$
Hence, for $t\ge t_0$ we have 
\begin{equation}\label{LB2}
p_{{B_R}}(t,x,y) \ge \left(\frac{c^*}{\sqrt{|{B_R}|}p_{t_0/2}(0)}\right)^2e^{-\lambda_1t_0} \p^x
\left(\tau_{{B_R}}>\frac{t_0}{2}\right) \p^y\left(\tau_{{B_R}}>\frac{t_0}{2}\right)
  e^{-\lambda_1t}.
\end{equation}
	Next, by Lemma \ref{GRTestimates}, there are constants $c_{10}=c_{10}(d), c_{11}=c_{11}(d)$ such that 
	
$$p_{t_0/2}(2R) \ge c_{10} t_0 \nu(2R) \exp\left(\frac{-c_{11} t_0}{V^2(2R)}\right).$$
	Since $\nu(2R)\ge\frac{C^*_{2R}}{(2R)^dV^2(2R)}$ by monotonicity and subadditivity of $V$,  
	$$ 
	\frac{c_7 C_R}{4} = \frac{t_0}{4 V^2(R)} \le \frac {t_0}{V^2(2R)}\le   {c_7 C_{R}}\le 1,
	$$
	so there are $c_{12}= c_{12}(d)$
	such that
	
	$$p_{t_0/2}(2R)\ge c_{12}\frac {C^*_{2R} C_R}{R^d} . $$	
By   Corolarry  \ref{p_t0} with  $r = (1/2)c_7 C_{R}$, 
	  $ p_{t_0/2}(0)\le \frac{c_{13}}{(C_R\underline{C}_R)^{1+d/\la}}  \frac 1{R^d}$ with $c_{13} =  c_{13}( d, \la)$.
	This implies that 
	
	\begin{equation}\label{ratio} \left(\frac{ p_{t_0/2}(2R)}{\sqrt{|{B_R}|}p_{t_0/2}(0)}\right)^2
	\ge c_{14}\left[(C_R\underline{C}_R)^{1+d/\la}{C^*_{2R}} C_R\right]^2\frac1{R^d}	\end{equation}
	with $c_{14} =  c_{14}( d, \la)$.
	
	On the other hand for all $ t\ge t_0$, 
	\begin{equation}\label{ratio1}p_{t\wedge V^2(R)}( x-y)\le p_{t_0/2}( 0) \le \frac{c_{13}}{( C_R\underline{C}_R)^{1+d/\la}}  \frac 1{R^d}\end{equation}. 

	Combining  (\ref{ratio}) and (\ref{ratio1})  we obtain for $ t\ge t_0$,
	
	\begin{equation}\label{ratio2}\left(\frac{ p_{t_0/2}(2R)}{\sqrt{|{B_R}|}p_{t_0/2}(0)}\right)^2\ge c_{15}{(C^*_{2R} C_R)^2}(C_{R}\underline{C}_R)^{3+3d/\la}p_{t\wedge V^2(R)}( x-y)	\end{equation}
	with $c_{15} =  c_{15}( d, \la)$.

Note also that $\lambda_1 t_0 = \lambda_1 V^2(R) c_7 C_R \le c$, where $c = c(d,\la)$. Hence, by (\ref{LB2}) and (\ref{ratio2}),  for $t\ge t_0$ we have 
\begin{eqnarray}
\nonumber
p_{{B_R}}(t,x,y) &\ge& \frac{ c_{16}(C^*_{2R})^2\underline{C}_R^{5+5d/\la}(C_R)^{{{23}+5d/\la}}}{H_R^4} 
\left(1\wedge\frac {C_{4R}^*}{2^{d+2}}\right)^2 \\
\label{LB101}
&& \times  \p^x\left(\tau_{{B_R}}>\frac{t_0}{2}\right) \p^y\left(\tau_{{B_R}}>\frac{t_0}{2}\right) e^{-\lambda_1t}p_{t\wedge V^2(R)}(x-y)
\end{eqnarray}	
with $c_{16} =  c_{16}( d, \la)$.
Due to Lemma \ref{L5a1}, and since $\mathcal{I}_R\ge c(d) C^*_{2R}$, we have the lower bound 
$$\p^x(\tau_{{B_R}}>t)\ge c_{17}\frac{C^*_{2R}}{H_R}\left(\frac{V(\delta_{B_R}(x))}{\sqrt{t}}\wedge1\right)$$
for $t\le c_{18}V^2(R)$ with $c_{17}= c_{17}(d)$ and $c_{18}= c_{18}(d)$. Recall that $t_0= c_7 V^2(R)C_R\le c_7 V^2(R)$. We may assume that the constant $c_7$ is smaller than $c_{18}$. Hence 
$$\p^x(\tau_{{B_R}}>t_0)\ge c_{17}\frac{C^*_{2R}}{H_R}\left(\frac{V(\delta_{B_R}(x))}{\sqrt{t_0}}\wedge1\right),$$
which combined with  (\ref{LB101}) yields for $t\ge t_0$, 
\begin{eqnarray}	
\nonumber
p_{{B_R}}\left(t,x,y\right)&\ge& \frac{ c_{19}(C^*_{2R})^{4}\underline{C}_R^{5+5d/\la}(C_R)^{{{23}+5d/\la}}}{H_R^6} 
\left(1\wedge\frac {C_{4R}^*}{2^{d+2}}\right)^2 p_{t\wedge V^2(R)}( x- y)\\
&& \times  \left(\frac{V(\delta_{B_R}(x))}{\sqrt{t}\wedge V(R)}\wedge1\right) \left(\frac{V(\delta_{B_R}(y))}{\sqrt{t}\wedge V(R)}\wedge1\right) e^{-\lambda_1t}\label{LB20}\end{eqnarray}
with $c_{19} =  c_{19}( d, \la)$.
If we set
$${\mathcal A}^*_R= \left(\frac{(\underline{C}_R)^{1+d/\la}(C_R)^{{{9}+d/\la}}}{H_R^2} 
\wedge  \frac{(C^*_{2R})^{4} \underline{C}_R^{5+5d/\la}(C_R)^{{{23}+5d/\la}}}{H_R^6} \right) 
\left(1\wedge\frac {C_{4R}^*}{2^{d+2}}\right)^2,$$
then combining (\ref{LB20})  with (\ref{LB100}) there is $c_{20}= c_{20}( d, \la)$ such that  

\begin{equation}	\label{last} p_{{B_R}}\left(t,x,y\right)\ge c_{20}{\mathcal A}^*_R\,  p_{t\wedge V^2(R)}( x- y)  
\left(\frac{V(\delta_{B_R}(x))}{\sqrt{t}\wedge V(R)}\wedge1\right) \left(\frac{V(\delta_{B_R}(y))}{\sqrt{t}\wedge V(R)}\wedge1\right) e^{-\lambda_1t},\end{equation}
for $t>0$.
 It is clear that ${\mathcal A}^*_R$ is nonincreasing in $R$, 
so the proof of the lower bound of $p_{{B_R}}$ is completed. 

To finish the proof we need to show a lower bound for $\p^x(\tau_{{B_R}}>t)$. By Lemma \ref{L5a1} it is clear that it is enough to consider $t\ge c_{21} V^2(R)$ for some $c_{21} = c_{21}(d) < 1$. Note that $t \wedge V^2(R) = c_{22} V^2(R)$ for some $c_{21} \le c_{22} \le 1$. We have

	$$p_{t\wedge V^2(R)}( x- y) \ge p_{c_{22} V^2(R)}(2R) \ge \frac{C_{2R} c_{22} V^2(R)}{V^2(2R) (2R)^d} \ge \frac{c_{21} C_{2R}}{4 (2R)^d}.  $$  Moreover for $|y|\le R/2$ and $t\ge t_0$,
$$\left(\frac{V(\delta_{B_R}(y))}{\sqrt{t}\wedge V(R)}\wedge1\right)\ge 1/2.$$ 
Integrating $p_{{B_R}}\left(t,x,y\right)$ over $ B_{R/2}$ with respect to $y$ and applying (\ref{last}) provides the desired bound for $\p^x(\tau_{{B_R}}>t)$. 
\end{proof}

\vskip 10pt

{\bf{ Acknowledgements.}} We thank prof. J. Zabczyk for communicating to us the problem of gradient estimates of the killed semigroup for jump processes.

\end{document}